\documentclass[12pt]{amsart}
\usepackage{verbatim}
\usepackage{fancyhdr}
\usepackage{graphicx}
\usepackage{amsmath,amsfonts,amsthm,amssymb}
\usepackage{subfigure}
\usepackage{color}
\usepackage{bm}
\usepackage{comment}
\usepackage{multirow}
\usepackage{enumerate}

\newtheorem{prop}{Proposition}[section]
\newtheorem{rem}{Remark}[section]
\newtheorem{exmp}{Example}

\numberwithin{equation}{section}

\newcommand{\jl}{j-\frac 12}
\newcommand{\jr}{j+\frac 12}

\newcommand{\bfu}{\mathbf{U}}
\newcommand{\bfv}{\mathbf{v}}
\newcommand{\jpl}{[\![}
\newcommand{\jpr}{]\!]}

\allowdisplaybreaks[4]

\begin{document}
\baselineskip 1.2pc

\title[OFDG for shallow water equations]{A well-balanced oscillation-free discontinuous Galerkin
method for shallow water equations}

\author{Yong Liu}
\address{LSEC, Institute of Computational Mathematics, Hua Loo-Keng Center for Mathematical Sciences,
Academy of Mathematics and Systems Science, Chinese Academy of Sciences, Beijing 100190, P.R. China.}
\thanks{Y. Liu's research is partially supported by the fellowship of China Postdoctoral Science Foundation
No. 2020TQ0343.}
\email{E-mail: yongliu@lsec.cc.ac.cn}

\author{Jianfang Lu}
\address{South China Research Center for Applied Mathematics
and Interdisciplinary Studies, South China Normal University, Canton,
Guangdong 510631, China.}
\thanks{J. Lu's research is partially supported by NSFC grant 11901213 and Guangdong Basic and
Applied Basic Research Foundation 2020B1515310021. }
\email{E-mail: jflu@m.scnu.edu.cn.}

\author{Qi Tao}
\address{Beijing Computational Science Research Center, Beijing 100193, China.}
\thanks{Q. Tao's research is supported in part by NSFC grants U1930402 and the fellowship of China Postdoctoral Science Foundation No. 2020TQ0030 }
\email{E-mail: taoqi@csrc.ac.cn}

\author{Yinhua Xia}
\address{University of Science and Technology of China, School of Mathematics, Hefei, Anhui 233026, P.R. China.}
\thanks{Y. Xia's research supported by the National Numerical Windtunnel Project NNW2019ZT4-B08
and a NSFC grant  No. 11871449.}
\email{E-mail: yhxia@ustc.edu.cn}

\subjclass[2010]{65M60}

\begin{abstract}
In this paper, we develop a well-balanced oscillation-free discontinuous Galerkin (OFDG)
method for solving the shallow water equations with a non-flat bottom topography.
One notable feature of the constructed scheme
is the well-balanced property, which preserves exactly the hydrostatic equilibrium solutions
up to machine error. Another feature is the non-oscillatory property, which
is very important in the numerical simulation when there exist some shock discontinuities.
To control the spurious oscillations, we construct an OFDG method with an extra damping term
to the existing well-balanced DG schemes proposed in \cite{XSCICP2006}.
With a careful construction of the damping term, the proposed method achieves both
the well-balanced property and non-oscillatory property simultaneously without compromising
any order of accuracy. We also present
a detailed procedure for the construction and a theoretical analysis for the preservation of
the well-balancedness property. Extensive numerical experiments including one- and two-dimensional space demonstrate that the proposed methods possess the desired properties without sacrificing any order of accuracy.
\end{abstract}

\keywords{Hyperbolic balance laws;
Oscillation-free discontinuous Galerkin method; Well-balanced scheme; Shallow water equations;}

\maketitle

\section{Introduction}
\label{sec_intro}

The nonlinear shallow water equations (SWEs) have wide applications in the modeling and
simulation of free surface flows in ocean and hydraulic engineering, including the dam break
and flooding problems, tidal flows in estuary and coastal water regions, etc.
They are also commonly used to
predict sea surface elevations and coastline changes due to hurricanes and ocean currents.
See e.g. \cite{BR2000, BSHG1995, GMFEM2019, Goutal1997, HCMK1999,
QCLP2018, XS2014} and the references therein.
The two-dimensional SWEs are reduced from the three-dimensional Navier-Stokes (NS) equations,
based on the fact that vertical length scale is much far less than the horizontal length scale
in many realistic situations such as the atmosphere and ocean. This allows us
to use SWEs instead of NS equations, for the reason that it could be very
expensive to simulate three-dimensional NS equations directly.

Since the SWEs are widely used in scientific research and engineering applications, it is very important to construct the robust and accurate numerical methods for solving the SWEs. There are two main numerical
difficulties in the computation of the SWEs. One is the preservation
of the well-balanced property. The traditional numerical methods may not
be able to balance the contribution of the source term and the flux gradient, and large numerical errors will occur on the coarse mesh or after a long time simulation.
A remedy to this difficulty is to use the refined mesh to reduce the numerical error,
which would tremendously increase the computational cost especially for the multidimensional
problems. Thus, it is very desirable to design the numerical schemes
that admit the equilibrium
solutions in which the flux gradient and the source term are exactly balanced,
and they are referred to as {\it well-balanced} schemes.
The main advantage of the well-balanced schemes is that they can be used to
resolve the small perturbations near the equilibrium state solution very precisely without
an excessively refined mesh. The well-balanced property is also referred to as the {\it  exact C-property},
which was first introduced by Bermudez and Vazquez in \cite{BV1994}. Since then, many
well-balanced numerical methods are constructed and studied in the framework of
the finite difference (FD) methods, finite volume (FV) methods and discontinuous
Galerkin (DG) methods, see e.g. \cite{ABBKP2004, BLMR2002, FMT2011, LeVeque1998, PS2001, XSJCP2005, XSJSC2006, ZCMI2001} and the references therein. We also refer to
the review paper \cite{XS2014} for a complete list of literatures on this topic.
Another difficulty is robustness of the numerical methods near the wet/dry front. Since the
SWEs are defined on the wet region only, then we need to deal with problems of moving boundaries. One feasible approach is to use the boundary-fitted mesh to track the
front \cite{Bokhove2005}. While
a more popular method is the thin layer technique, which maintains a very thin layer in the
dry region so that the SWEs are also defined on it. Then the difficulty is converted into
 the positivity-preserving of the water heights during the simulation. There exist
a vast amount of the positivity-preserving FV schemes and DG schemes,
see e.g. \cite{ABBKP2004, BEKP2011, BKWD2009, CMP2007, KL2012}.
Based on the approach developed in \cite{ZS2010_scalar, ZS2010_system},
Xing et al. constructed the positivity-preserving FD
schemes, FV schemes and DG schemes for the SWEs without
destroying the high order accuracy, conservation and well-balanced property
\cite{XS2011, XZ2013, XZS2010}. Very
recently, Wen et al. in \cite{WDGX2020} developed an entropy stable and positivity-preserving
well-balanced DG method to compute the SWEs.

In this work, we propose a high order well-balanced oscillation-free DG (OFDG) method for solving the SWEs.
The constructed numerical method is based on the well-balanced DG schemes proposed in
\cite{XSCICP2006}, which only used the simple source term approximation but a careful
construction of the numerical fluxes employing the idea of hydrostatic reconstruction
in \cite{ABBKP2004}. To treat the wet/dry front, we adopt the OFDG
method developed in \cite{LLSSISC2021, LLSSINUM2021}.
The OFDG method can not only control the spurious oscillations, but also
maintain the high order accuracy, conservation and superconvergence properties.
In this paper, we give a simple analysis of maintaining the
well-balanced property, which indicates that the original OFDG method is consistent with
this property well. We test many benchmark problems and obtain the satisfactory numerical results.
 This strongly demonstrates the effectiveness and robustness of our method.

The organization of this paper is  as follows. In Section \ref{sec_1d_swe}, we consider the
one-dimensional SWEs and construct the corresponding well-balanced OFDG schemes.
A semi-discrete analysis of preserving the well-balanced property is also given.
In Section \ref{sec_md_swe}, we extend the one-dimensional results to the two-dimensional SWEs.
We conduct a numerical investigation of the proposed
algorithm, including the accuracy tests and well-balanced property preserving in Section \ref{sec_numeric}.
Some concluding remarks are given in Section \ref{sec_sum}.

\section{One-dimensional well-balanced OFDG schemes}
\label{sec_1d_swe}

In this section, we consider the one-dimensional shallow water equation given as follows:
\begin{align} \label{SWE_1d}
\left\{\begin{aligned}
& h_t+(hu)_x=0, \\
& (hu)_t+\Big(hu^2+\frac 12 g h^2 \Big)_x = -g hb_x,
\end{aligned} \right.
\end{align}
where $u$ is the velocity of the fluid, $h$ denotes the water height, $b(x)$ represents
the bottom topography and $g$ is the gravitational constant.
This model admits steady state solutions, in which the flux gradient is exactly
balanced by the source term. In particular, people are interested in the still water
stationary solutions, which are given by
\begin{align}\label{stat_eq_1d}
hu = 0,\quad h + b =\text{constant}.
\end{align}
First, we assume that the discretization of  the computational
domain is given by cells $I_{j}=[x_{\jl},x_{\jr}],~j=1,\cdots, N$. We denote the cell length as $\Delta x_{j}=x_{\jr}-x_{\jl},$ and $\Delta x
=\max\limits_j\Delta x_{j}.$
The finite element space $V_h^k$ is defined as follows:
\begin{align}
V_h^k:=\{v: v|_{I_j}\in P^k(I_j), \, j=1,\cdots, N\},
\end{align}
where $P^k(I_j)$ denotes the polynomials with degree at most $k$ in $I_j$ .
Our goal is to construct the well-balanced OFDG scheme for \eqref{SWE_1d}. The idea of the OFDG method in \cite{LLSSISC2021, LLSSINUM2021}
is to add the suitable numerical damping terms to the conventional DG schemes in order
to control the spurious oscillations. Therefore, to construct the well-balanced OFDG schemes
we also add the numerical damping terms to the well-balanced DG schemes proposed in
\cite{XSCICP2006}. Then, we can define the semi-discrete well-balanced OFDG scheme for \eqref{SWE_1d}:
Find $\bfu_h(\cdot,t)\in [V_h^k]^2$ such that for any $\bfv_h \in [V_h^k]^2$ we have
\begin{equation}
\label{wbofdg_scheme}
\begin{split}
\int_{I_j}(\bfu_h)_t\cdot \bfv_h \, dx = &\int_{I_j} \mathbf{F}(\bfu_h)\cdot (\bfv_h)_x \, dx
-\widehat{\mathbf{F}}^l_{\jr}\cdot(\bfv_h)_{\jr}^- + \widehat{\mathbf{F}}^r_{\jl}\cdot(\bfv_h)_{\jl}^+ \\
 +\int_{I_j}&\mathbf{S}(\bfu_h,b_h) \cdot \bfv_h \, dx- \sum_{\ell=0}^{k}
 \frac{\sigma_j^\ell(\bfu_h)}{\Delta x_j}\int_{I_j}(\widetilde{\bfu}_h-P_h^{\ell-1}
 \widetilde{\bfu}_h)\cdot\bfv_h \, dx,
\end{split}
\end{equation}
where $\bfu_h = (h_h, (hu)_h)^T$ is the approximation of the unknown $\bfu = (h, hu)^T$, $b_h(x)$ is the $L^2$ projection of function $b(x)$ into $V_h^k$, $\mathbf{F}(\bfu) = \big(hu, hu^2 + gh^2/2\big)^T$ is the flux function, and $\mathbf{S}(\bfu,b) = (0,-ghb_x)^T$ is the source term.
The left numerical flux $\widehat{\mathbf{F}}^l_{\jr}$ and the right numerical flux $\widehat{\mathbf{F}}^r_{\jl}$
are defined in \cite{XSCICP2006}.
The choices of the numerical fluxes are crucial to obtain the
well-balanced property. There are two choices to define the left and right fluxes in \cite{XSCICP2006}.
In this paper, we consider {\em Choice B} in \cite{XSCICP2006}, that is, after computing boundary
values $\big( \bfu_h\big)_{\jr}^{\pm}$, we set
\begin{align} \label{eqn_bdy_value_1}
\big(h_h\big)_{\jr}^{\ast, \pm} = \max\Big(0,~\big(h_h + b_h \big)_{\jr}^\pm - \big(b_h\big)_{\jr}^\ast
\Big),\, \big(b_h \big)^\ast_{\jr} = \max\big( \big(b_h)_{\jr}^+, \big(b_h\big)_{\jr}^- \big) \, .
\end{align}
The left and right values of $\bfu_h$ are redefined as follows:
\begin{align} \label{eqn_bdy_value_2}
\big( \bfu_h\big)_{\jr}^{\ast,\pm} = \Big( \big(h_h \big)_{\jr}^{\ast,\pm}, \big( (hu)_h
\big)_{\jr}^{\pm} \Big)^T.
\end{align}
Then the left and right fluxes $\widehat{\mathbf{F}}^l_{\jr}$ and $\widehat{\mathbf{F}}^r_{\jl}$
are given by:
\begin{align} \label{eqn_flux_1d}
\begin{aligned}
&\widehat{\mathbf{F}}^l_{\jr} = \widehat{\mathbf{F}}\Big(\big( \bfu_h \big)_{\jr}^{\ast,-},
\big( \bfu_h \big)_{\jr}^{\ast,+}\Big)
+\Big(0,\, \frac{g}{2} \big( (h_h^2 )_{\jr}^- - ( h_h^2 )_{\jr}^{\ast,-} \big) \Big)^T, \\
&\widehat{\mathbf{F}}^r_{\jl} = \widehat{\mathbf{F}}\Big(\big( \bfu_h \big)_{\jl}^{\ast,-},
\big( \bfu_h \big)_{\jl}^{\ast,+} \Big)
+\Big(0,\, \frac{g}{2} \big( ( h_h^2 )_{\jl}^+ -( h_h^2)_{\jl}^{\ast,+}\big) \Big)^T \, .
\end{aligned}
\end{align}
Here, $\widehat{\mathbf{F}}$ is the numerical flux and a simply choice is the Lax-Friedrichs flux
(see e.g. \cite{Shu2009}). The last term of the right-hand side of \eqref{wbofdg_scheme} is
the artificial damping terms which were introduced in \cite{LLSSISC2021,LLSSINUM2021}.
$P_h^{\ell},\, \ell\geq 0$ is the standard $L^2$ projection for vector functions, and defined as follows: for $\forall\mathbf{w}$, $P_h^{\ell}\mathbf{w}\in [V_h^{\ell}]^2$ such that
\begin{align}
\int_{I_j}(P_h^{\ell}\mathbf{w} -\mathbf{w}) \cdot \bfv_h \, dx =0, \quad \forall \,
\bfv_h \in [P^{\ell}(I_j)]^{2}.
\end{align}
Here, we define $P_h^{-1}=P_h^{0}$ and follow the idea in \cite{LLSSISC2021}, the damping
coefficients $\sigma_j^{\ell} \geq 0$ are given by
\begin{align} \label{eqn_damping_coef_1d}
\sigma_j^{\ell} = \frac{2(2\ell+1)}{2k-1} \frac{(\Delta x)^{\ell}}{\ell !} \max_{s=1,2}
\left(\jpl \partial_{x}^{\ell} V_s\jpr _{\jl}^2+\jpl \partial_{x}^{\ell} V_s\jpr_{\jr}^2\right)^{\frac 12},
\end{align}
where $\jpl v \jpr_{\jr}=v\big(x_{\jr}^+ \big) - v\big(x_{\jr}^- \big)$ denotes the jump of
$v$ at $x=x_{\jr}$. The variables $\partial_x^{\ell}\mathbf{V} = \big(\partial_x^\ell V_1,
\partial_x^{\ell} V_2\big)^T$ are given by $\partial_x^{\ell}\mathbf{V}=\mathbf{R}^{-1}
\partial_x^{\ell}\bfu_h$, $0 \leq \ell \leq k$, and $\mathbf{R}^{-1}$ is the matrix derived from
the characteristic decomposition $\mathbf{F}'(\overline{\bfu_h}_{\jr}) = \mathbf{R}\Lambda
\mathbf{R}^{-1}$, and $(\overline{~\cdot~})_{\jr}$
stands for some average on $x_{\jr}$, such as the arithmetic mean or the Roe average.
For one-dimensional shallow water equations, we have $\mathbf{R}^{-1}$ defined as
\begin{align}
\mathbf{R}^{-1}=
\begin{pmatrix}
c+u & -1\\
c-u & 1
\end{pmatrix},
\end{align}
where $c = \sqrt{gh}$. We apply the extra damping terms to the variables $\widetilde{\bfu}_h$ instead of $\bfu_h$ to guarantee the well-balanced property, where $\widetilde{\bfu}_h$ is defined as:
\begin{align}
\widetilde{\bfu}_h = (h_h + b_h, (hu)_h)^T.
\end{align}

\begin{prop}
The OFDG scheme defined in \eqref{wbofdg_scheme} is well-balanced for still water
stationary state \eqref{stat_eq_1d} for shallow water equations.
\end{prop}
\begin{proof}
We define the residual
\begin{equation}
\begin{split}
Res = &\int_{I_j} \mathbf{F}(\bfu_h)\cdot (\bfv_h)_x \, dx -\widehat{\mathbf{F}}^l_{\jr}\cdot(\bfv_h)_{\jr}^- + \widehat{\mathbf{F}}^r_{\jl}\cdot(\bfv_h)_{\jl}^+ \\
& +\int_{I_j}\mathbf{S}(\bfu_h,b_h) \cdot \bfv_h \, dx- \sum_{\ell=0}^{k}\frac{\sigma_j^\ell(\bfu_h)}{\Delta x_j}\int_{I_j}(\widetilde{\bfu}_h-P_h^{\ell-1}\widetilde{\bfu}_h)\cdot\bfv_h \, dx
\end{split}
\end{equation}
From Proposition 3.1 in \cite{XSCICP2006}, the residual $Res$ without the extra damping term in \eqref{wbofdg_scheme} for still water would reduce to zero. Thus, we only need to show the damping
term vanishes for still water. It is easy to see when $h_h+b_h = constant$ and $(hu)_h=0$,
i.e. $\widetilde{\bfu}_h = (constant,0)^T$, we have
\begin{align}
\int_{I_j}(\widetilde{\bfu}_h-P_h^{\ell-1}\widetilde{\bfu}_h)\cdot\bfv_h \, dx = 0, \quad
\forall \, \bfv_h \in [P^{k}(I_j)]^2.
\end{align}
Therefore, the residual $Res$ in \eqref{wbofdg_scheme} is zero for the still water state \eqref{stat_eq_1d}.
\end{proof}

\begin{rem}
In \cite{XSJCP2006}, Xing and Shu proposed a well-balanced DG scheme for the SWEs. They
decomposed the source term into a sum of three terms to achieve well-balanced property.
We can also follow the above procedure similarly to design the well-balanced OFDG scheme
based on the DG scheme in \cite{XSJCP2006}. Numerically this treatment makes little
difference comparing to the one in \eqref{wbofdg_scheme}, thus we only focus ourselves
on the scheme \eqref{wbofdg_scheme} throughout this paper.
\end{rem}

\section{Two dimensional well-balanced OFDG schemes}
\label{sec_md_swe}

This section will extend the one-dimensional well-balanced OFDG scheme
\eqref{wbofdg_scheme} to the two dimensional space.
Now, we consider the two-dimensional shallow water equations:
\begin{align}
\label{SWE_2d}
\left\{\begin{aligned}
& h_t+(hu)_x+(hv)_y = 0 \, ,\\
& (hu)_t+\Big( hu^2+\frac 12 g h^2 \Big)_x+(huv)_y = -ghb_x \, ,
\quad (x,y) \in \Omega,~t\in (0,T],\\
& (hv)_t +(huv)_x + \Big( hv^2+\frac 12 g h^2 \Big)_y = -ghb_y \, ,
\end{aligned}\right.
\end{align}
where $h$ is the water height, $(u, v)$ is the velocity of the fluid,
$b(x, y)$ represents the bottom topography and $g$ is the gravitational constant. The still water
stationary solutions are given by
\begin{align} \label{stat_eq_2d}
hu =  hv = 0, \quad h + b = \text{constant}.
\end{align}
To obtain the OFDG scheme, firstly, we assume that a regular partition $\mathcal{T}_h$
of $\Omega$ is given. For each element $K\in\mathcal{T}_h$, $\Delta_K = \text{diam} K$,
$\Delta_{\mathcal{T}_h} = \max\limits_{K\in\mathcal{T}_h} \Delta_K$.
Then, the OFDG scheme for \eqref{SWE_2d} is defined as follows:
Seek $\bfu_h(\cdot,t)\in [V_h^k]^3$ such that  $\forall \, \bfv_h \in [V_h^k]^3$ we have
\begin{equation}
\begin{split}
\label{wbofdg_scheme_2d}
& \int_K (\bfu_h)_t \cdot \bfv_h \,dx dy \\
 =& \int_{K}\mathbf{F}(\bfu_h)\cdot (\bfv_h)_x +\mathbf{G}(\bfu_h)\cdot (\bfv_h)_y \, dx dy\\
& - \int_{\partial_K}(\widehat{\mathbf{F}}^{\partial_K}(\bfu_h)n_1
+ \widehat{\mathbf{G}}^{\partial_K}(\bfu_h)n_2)\cdot \bfv_h \, ds
+ \int_{K}\mathbf{S}(\bfu_h,b_h)\cdot \bfv_h \, dx dy \\
&  -\sum_{\ell=0}^{k}\frac{\sigma_{K}^\ell(\bfu_h)}{\Delta_K}
\int_{K}(\widetilde{\bfu}_h-P_h^{\ell-1}\widetilde{\bfu}_h)\cdot\bfv_h \, dx dy ,
\end{split}
\end{equation}
where $\bfu_h=(h_h,(hu)_h,(hv)_h)^T$, $\widetilde{\bfu}_h=(h_h+b_h,(hu)_h,(hv)_h)^T$,
$\mathbf{F}(\bfu)=(hu,hu^2+ gh^2/2,huv)^T$, $\mathbf{G}(\bfu)=(hv,huv,hv^2+ gh^2/2)^T$,
$\mathbf{S}(\bfu)=(0,-ghb_x,-ghb_y)^T$.
$\widehat{\mathbf{F}}^{\partial_K}$ and $\widehat{\mathbf{G}}^{\partial_K}$ are numerical fluxes
obtained by the same procedure as one-dimensional case in Section \ref{sec_1d_swe}.
$\mathbf{n}=(n_1,n_2)^T$ is the unit outward
normal with respect to $\partial_K$.
The damping coefficients are defined by
\begin{align}\label{damping_coef_md}
\sigma_{K}^\ell=\frac{2(2\ell+1)}{2k-1}\frac{(\Delta_{\mathcal{T}_h})^{\ell}}{\ell !}
\max_{1\leq s \leq 3}\sum_{|\bm{\alpha}|=\ell} \bigg(\frac{1}{N_e}\sum_{\bm{v}\in \partial_K}
\left(\jpl\partial^{\bm{\alpha}}V_s\jpr\Big|_{\bm{v}}\right)^2\bigg)^{\frac 12},
\end{align}
where $\bm{\alpha}=(\alpha_1,\alpha_2)$ is the multi-index notation with non-negative integers of the length
\begin{align}
|\bm{\alpha}|=\alpha_1+\alpha_2,
\end{align}
and $\partial^{\alpha}w=\partial_x^{\alpha_1}\partial_y^{\alpha_2} w$. We also use
the jump of the characteristic variables $\mathbf{V}=\mathbf{R}^{-1}\bfu$ to define the
$\sigma_K^\ell$. The matrix $\mathbf{R}$ is obtained from the characteristic decomposition
\begin{align}
n_1\mathbf{F}'(\overline{\bfu_h})+n_2\mathbf{G}'(\overline{\bfu_h})=\mathbf{R}
\bm{\Lambda}\mathbf{R}^{-1}
\end{align}
on the element interface, and $\overline{\bfu_h}$ is the mean average or Roe average on the element
interface. $[\![ w ]\!]\big|_{\bm{v}} $ denotes the jump of the function $w$ on the vertex $\bm{v} $.
 $N_e$ is the number of edges of the element $K$ and $\bm{v} \in \partial_{K}$ are the vertices of $K$.
For illustration purpose, we consider the two-dimensional case as follows.
\begin{figure}[!htb]
\centering
\caption{\label{fig_jump_in_sigma}  Graph for the illustration of the jumps in
$\sigma_K^\ell$ defined in \eqref{damping_coef_md}.} \bigskip\bigskip\bigskip
\includegraphics[width=0.5\textwidth, height = 0.3\textwidth]{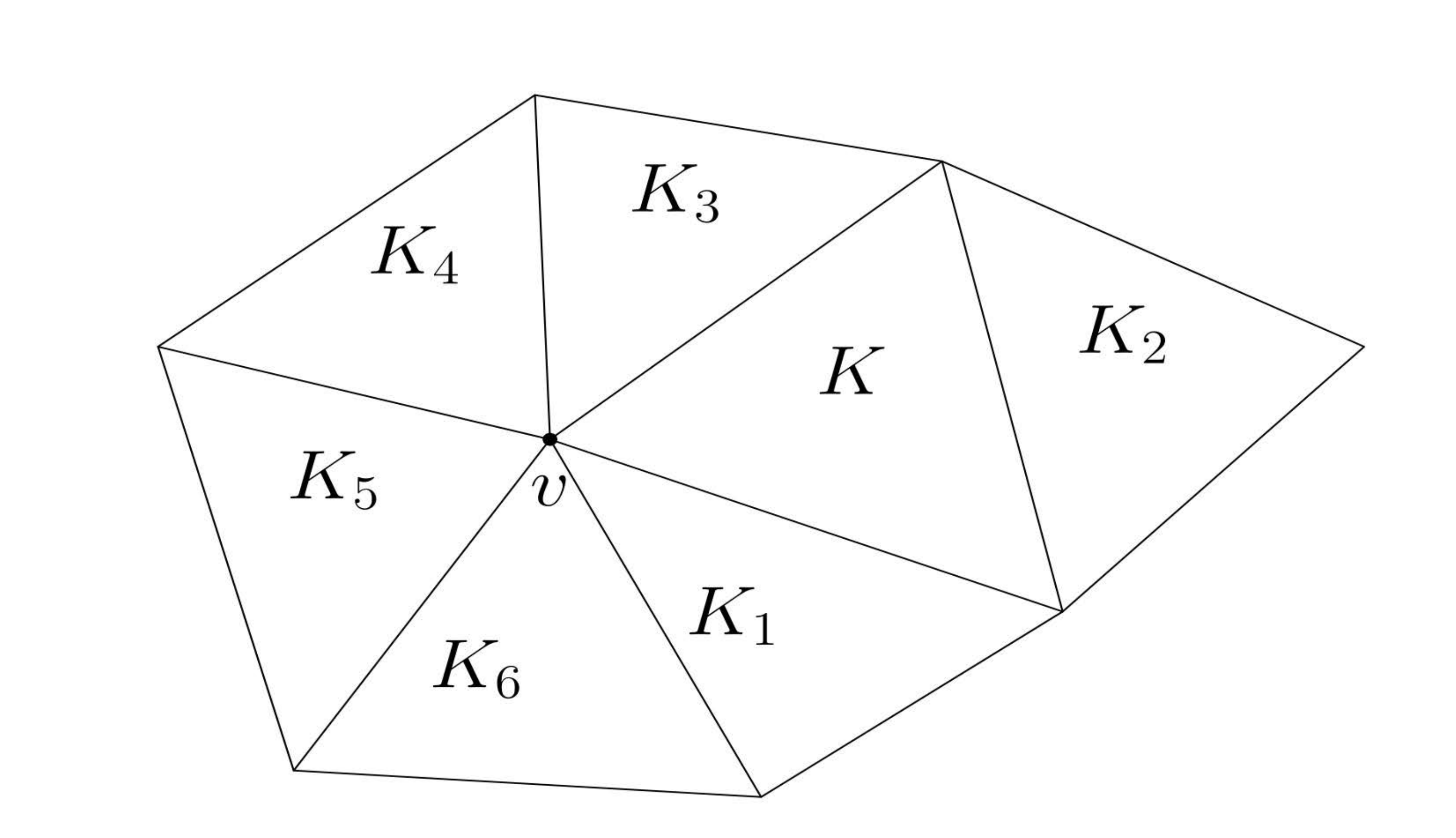}
\end{figure}

In Figure \ref{fig_jump_in_sigma}, we can see that $N_e=3$ for the element $K$ and $K_1, K_2, K_3$ are its adjacent neighbors, then we have:

\[
\Big( [\![ \partial^{\bm{\alpha}} w ]\!] \Big|_{\bm{v}} \Big)^2
= \big[ (\partial^{\bm{\alpha}} w\big|_K - \partial^{\bm{\alpha}} w\big|_{K_1})^2
+ (\partial^{\bm{\alpha}} w\big|_K - \partial^{\bm{\alpha}} w\big|_{K_3})^2 \big] \Big|_{\bm{v}}.
\]
For more details we refer the reader to \cite{LLSSINUM2021}.
Throughout this paper, we take $\mathbf{R}^{-1}$ as follows
\begin{align}
\mathbf{R}^{-1}=
\begin{pmatrix}
c+un_1+vn_2 & n_1 & n_2\\
2(u n_2 -v n_1) & -2n_2 & 2n_1 \\
c+un_1+vn_2& -n_1 & -n_2
\end{pmatrix},
\end{align}
where $c=\sqrt{gh}$.

\medskip

The well-balanced property of the scheme \eqref{wbofdg_scheme_2d} is as follows. The proof
 is similar to one-dimensional case, thus we omit it here.
\begin{prop}
The OFDG scheme \eqref{wbofdg_scheme_2d} preserves the well-balanced property for the
still water stationary
state \eqref{stat_eq_2d}.
\end{prop}

\section{Numerical Tests}
\label{sec_numeric}

We test some one- and two-dimensional numerical examples
 and some benchmark problems to demonstrate the good performance
of the proposed scheme in this section. The time discretization
method in all numerical tests is the fourth order Runge-Kutta (RK4) method given in the Butcher tableau
in Table \ref{tab_rk4}, and the piecewise $P^2$ polynomial space is used unless otherwise
specified. The gravitational constant $g=9.812m/s^2$.

\begin{table}[!ht]
\caption{ \label{tab_rk4} The Butcher tableau of the RK4. }
\begin{tabular}{c|cccc}
0 &  &  &  & \\
1/2 & 1/2  &  &  &  \\
1/2 & 0 & 1/2 &  &  \\
1 & 0 & 0 & 1 &  \\ \hline
 & 1/6 & 1/6 & 1/3 & 1/6
\end{tabular}
\end{table}

\subsection{One-dimensional Problems}
\label{sec_numeric_1d}

\begin{exmp}\label{examp1}
In this example, we consider two different bottom functions, one is smooth and another is discontinuous,
to verify the proposed OFDG scheme maintain the well-balanced property
over both bottoms.
A smooth bottom is given by
\begin{align}\label{numequ1d_1}
b(x)=5e^{-\frac 25 (x-5)^2}, 0\leq x \leq 10,
\end{align}
and a discontinuous bottom is given by:
\begin{align}\label{numequ1d_2}
b(x)=\left\{\begin{array}{cc}
4 & \text{ if } 4 \leq x \leq 8,\\
0 & \text{ otherwise},
\end{array}\right.
\end{align}
The initial data satisfy the following stationary state:
\begin{align*}
h+b=10,\quad hu=0.
\end{align*}
\end{exmp}
We compute both solutions until $t=0.5$ on the uniform mesh with $N=200$.
In the simulations, we adopt the different precisions as shown in Table \ref{tab1} to verify the well-balanced property.  The errors for the water height $h$ and the discharge $hu$ in $L^1$, $L^2$
and $L^{\infty}$ norm are shown in Table \ref{tab1}.
From Table \ref{tab1}, it is clearly to see that our scheme preserves the steady state in the
round-off error level.

\begin{table}[htb]
\centering
\caption{$L^1$, $L^2$ and $L^\infty$ errors for the stationary solution with different precisions in Example \ref{examp1}. \label{tab1}}
\resizebox{\textwidth}{!}{
\begin{tabular}{|c|c|cc|cc|cc|}
  \hline
   &  & \multicolumn{2}{c|}{$L^1$ error} & \multicolumn{2}{c|}{$L^2$ error} & \multicolumn{2}{c|}{$L^\infty$ error}\\ \cline{3-8}
   & precision & $h$ & $hu$ & $h$ & $hu$& $h$ & $hu$ \\ \hline
  \multirow{3}{*}{smooth}
   & single &  1.372E-05           &  6.251E-05                    &  1.424E-05           &  8.011E-05           &  2.193E-05
           &  2.484E-04         \\\cline{2-8}
   & double & 2.909E-14 & 8.752E-14           & 2.953E-14 & 1.091E-13           & 4.441E-14         & 2.599E-13         \\\cline{2-8}
   & quadruple &  2.511E-32    &  7.277E-32    &  2.587E-32    &  8.911E-32    &  4.314E-32
    &  2.945E-31
   \\\hline
  \multirow{3}{*}{nonsmooth}
  & single &  1.376E-07    &  9.211E-06    &  3.351E-07    &  3.024E-05    &  1.431E-06
    &  1.902E-04        \\\cline{2-8}
   & double &  5.611E-16    &  7.560E-14    &  9.625E-16    &  1.258E-13    &  3.553E-15
    &  6.733E-13         \\\cline{2-8}
   & quadruple &  2.385E-33    &  9.419E-33    &  2.535E-33    &  2.051E-32    &  4.622E-33 &  1.187E-31         \\
  \hline
\end{tabular}
}
\end{table}

\begin{exmp}
\label{examp9}
In this example we consider a still water steady state with non-flat bottom containing a wet/dry interface.
The bottom topography is given by \cite{LMAdv2009,XZS2010}:
\begin{align*}
b(x) = \max(0,0.25 - 5(x-0.5)^2), \quad 0 \leq x \leq 1.
\end{align*}
The initial conditions are the stationary state solution:
\begin{align*}
h + b = \max(0.2, b), \quad hu = 0.
\end{align*}
The periodic boundary conditions are considered.
\end{exmp}
Note that in this case $h+b$ is no longer a
constant function. In Figure \ref{fig_examp9}, we plot the
surface level $h + b$ and the bottom $b$. Since the water height $h=0$ if $0.4\leq x \leq 0.6$, then the numerical solution of $h$ easily becomes negative in this region. Thus the
positive-preserving limiter \cite{XZS2010} should be applied in this example. In the simulations, we use 200 uniform cells and compute the solution until $t = 0.5$.
We also use the different precisions to verify that $L^1$, $L^2$ and $L^\infty$ errors are
at the level of round-off error, and present results in Table \ref{tab_examp9}. From Figure \ref{fig_examp9} and Table \ref{tab_examp9}, we can see that the well-balanced OFDG scheme \eqref{wbofdg_scheme} combined with the positive-preserving limiter does not destroy the well-balanced property.
\begin{table}[htb]
\centering
\caption{$L^1$, $L^2$ and $L^\infty$ errors for different precisions in Example \ref{examp9}. \label{tab_examp9}}
\resizebox{\textwidth}{!}{
\begin{tabular}{|c|cc|cc|cc|}
  \hline
     & \multicolumn{2}{c|}{$L^1$ error} & \multicolumn{2}{c|}{$L^2$ error} & \multicolumn{2}{c|}{$L^\infty$ error}\\ \cline{2-7}
    precision & $h$ & $hu$ & $h$ & $hu$& $h$ & $hu$ \\ \hline
    single &  1.642E-08    &  7.765E-08    &  2.235E-08    &  1.499E-07    &  9.220E-08    &  8.003E-07\\ \hline
    double &  2.113E-15    &  1.160E-15    &  2.413E-15    &  1.514E-15    &  3.553E-15    &  5.500E-15        \\\hline
    quadruple &  8.253E-33    &  8.992E-33    &  1.046E-32    &  1.452E-32    &  1.914E-32    &  8.436E-32
   \\
   \hline
\end{tabular}
}
\end{table}

\begin{figure}[htb]
\centering
\caption{ \label{fig_examp9} The surface level $h+b$ and the bottom $b$ for the stationary flow in Example \ref{examp9}.}
\includegraphics[width=0.5\textwidth]{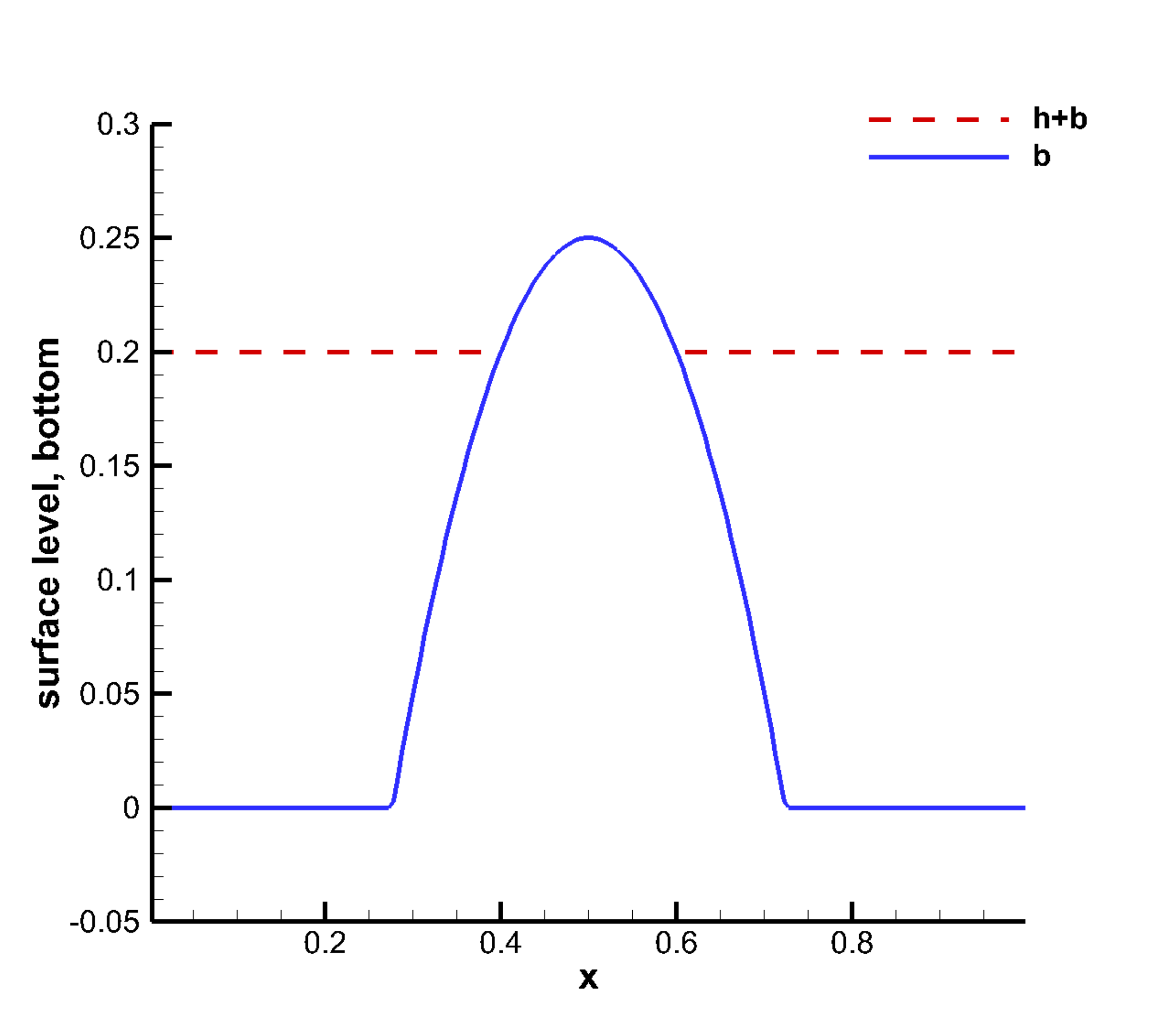}
\end{figure}

\begin{exmp} \label{examp2}
The orders of accuracy for the well-balanced OFDG scheme \eqref{wbofdg_scheme} will be tested in this example. We consider the periodic boundary conditions and take the smooth bottom function:
\begin{align*}
b(x) = \sin^2(\pi x),~~ x\in [0,1].
\end{align*}
The initial data are given by
\begin{align*}
 h(x,0)=5+e^{\cos(2\pi x)}, ~~ (hu)(x,0)=\sin(\cos(2\pi x)), ~~ x\in [0,1].
\end{align*}
\end{exmp}
Since we do not explicitly know  the exact solutions of this problem, we adopt the a posteriori error
$\|\bfu_h-\bfu_{\frac h2}\|$ as the numerical errors to compute the convergence rate. The terminal time is set to $t=0.1$ such that the solution is still smooth. We use $P^1$, $P^2$ and $P^3$ piecewise polynomials as finite element spaces. From Tables \ref{tab_1d_SW_smooth_h}-\ref{tab_1d_SW_smooth_hu}, the optimal convergence rate of all variables can be observed in each case.
\begin{table}[htb]\small
\caption{\label{tab_1d_SW_smooth_h} $h$'s numerical errors and orders
in Example \ref{examp2}. } \centering
\medskip
\begin{tabular}{|c|r||cc|cc|cc|}  \hline
 & $N$ & $L^1$ error & order & $L^2$ error & order & $L^\infty$ error & order \\ \hline \hline
 \multirow{6}{0.6cm}{$P^1$}
    &   10    &  5.242E-02    &   --    &  6.940E-02    &   --    &  1.964E-01    &   --   \\
    &   20    &  1.461E-02    &   1.843    &  2.529E-02    &   1.456    &  8.430E-02    &   1.220   \\
    &   40    &  3.068E-03    &   2.252    &  6.031E-03    &   2.068    &  2.343E-02    &   1.847   \\
    &   80    &  5.806E-04    &   2.401    &  1.262E-03    &   2.257    &  7.800E-03    &   1.587   \\
    &  160    &  1.050E-04    &   2.468    &  2.202E-04    &   2.519    &  1.495E-03    &   2.384   \\
    &  320    &  2.220E-05    &   2.241    &  4.343E-05    &   2.342    &  3.045E-04    &   2.295   \\\hline
 \multirow{6}{0.6cm}{$P^2$}
    &   10    &  1.028E-02    &   --    &  1.952E-02    &   --    &  6.027E-02    &   --  \\
    &   20    &  1.999E-03    &   2.362    &  4.547E-03    &   2.102    &  1.938E-02    &   1.637   \\
    &   40    &  2.353E-04    &   3.086    &  6.390E-04    &   2.831    &  4.108E-03    &   2.238   \\
    &   80    &  2.146E-05    &   3.455    &  6.934E-05    &   3.204    &  6.373E-04    &   2.689   \\
    &  160    &  2.071E-06    &   3.373    &  6.798E-06    &   3.350    &  9.100E-05    &   2.808   \\
    &  320    &  2.277E-07    &   3.185    &  7.575E-07    &   3.166    &  1.267E-05    &   2.845   \\ \hline
 \multirow{6}{0.6cm}{$P^3$}
    &   10    &  3.372E-03    &  --   &  7.495E-03    &   --    &  2.794E-02    &   --\\
    &   20    &  4.070E-04    &   3.050    &  1.054E-03    &   2.830    &  5.064E-03    &   2.464\\
    &   40    &  2.815E-05    &   3.854    &  9.152E-05    &   3.526    &  6.525E-04    &   2.956\\
    &   80    &  1.237E-06    &   4.508    &  4.393E-06    &   4.381    &  4.369E-05    &   3.901\\
    &  160    &  6.439E-08    &   4.264    &  2.506E-07    &   4.132    &  3.245E-06    &   3.751\\
    &  320    &  3.778E-09    &   4.091    &  1.523E-08    &   4.040    &  1.931E-07    &   4.071\\ \hline
\end{tabular}
\end{table}

\begin{table}[htb]\small
\caption{\label{tab_1d_SW_smooth_hu} $hu$'s numerical errors and orders
in Example \ref{examp2}. } \centering
\medskip
\begin{tabular}{|c|r||cc|cc|cc|}  \hline
 & $N$ & $L^1$ error & order & $L^2$ error & order & $L^\infty$ error & order \\ \hline \hline
 \multirow{6}{0.6cm}{$P^1$}
    &   10    &  2.571E-01    &   --    &  3.737E-01    &   --    &  8.432E-01    &   --   \\
    &   20    &  8.741E-02    &   1.557    &  1.492E-01    &   1.325    &  5.076E-01    &   0.732   \\
    &   40    &  2.537E-02    &   1.785    &  5.202E-02    &   1.520    &  2.033E-01    &   1.320   \\
    &   80    &  4.726E-03    &   2.424    &  1.077E-02    &   2.272    &  6.570E-02    &   1.629   \\
    &  160    &  8.391E-04    &   2.494    &  1.885E-03    &   2.514    &  1.290E-02    &   2.349   \\
    &  320    &  1.763E-04    &   2.250    &  3.721E-04    &   2.341    &  2.643E-03    &   2.287   \\\hline
 \multirow{6}{0.6cm}{$P^2$}
    &   10    &  8.069E-02    &   --    &  1.568E-01    &   --    &  5.599E-01    &   --   \\
    &   20    &  1.414E-02    &   2.512    &  3.342E-02    &   2.230    &  1.379E-01    &   2.021   \\
    &   40    &  1.941E-03    &   2.865    &  5.386E-03    &   2.634    &  3.297E-02    &   2.065   \\
    &   80    &  1.802E-04    &   3.430    &  5.939E-04    &   3.181    &  5.416E-03    &   2.606   \\
    &  160    &  1.703E-05    &   3.403    &  5.818E-05    &   3.352    &  7.835E-04    &   2.789   \\
    &  320    &  1.864E-06    &   3.192    &  6.481E-06    &   3.166    &  1.094E-04    &   2.840   \\
\hline
 \multirow{6}{0.6cm}{$P^3$}
    &   10    &  2.888E-02    &   --    &  6.340E-02    &   --    &  2.117E-01    &   --   \\
    &   20    &  3.487E-03    &   3.050    &  9.055E-03    &   2.808    &  4.477E-02    &   2.241   \\
    &   40    &  2.422E-04    &   3.848    &  7.915E-04    &   3.516    &  5.728E-03    &   2.967   \\
    &   80    &  1.066E-05    &   4.506    &  3.814E-05    &   4.375    &  3.895E-04    &   3.878   \\
    &  160    &  5.570E-07    &   4.258    &  2.166E-06    &   4.138    &  2.877E-05    &   3.759   \\
    &  320    &  3.252E-08    &   4.098    &  1.314E-07    &   4.043    &  1.705E-06    &   4.076   \\\hline
\end{tabular}
\end{table}

\begin{exmp}\label{examp3}
A small perturbation of a steady-state
water is considered to examine the capability of capturing the small perturbation of the proposed scheme. This test example was proposed by LeVeque \cite{LeVeque1998}.
The bottom topography is given by:
\begin{align*}
b(x) = \left\{ \begin{aligned}
0.25\Big(\cos\big(10\pi(x-1.5)\big)+1\Big) & \text{ if~~ } 1.4 \leq x \leq 1.6,\\
0\quad\quad & \text{ otherwise.}
\end{aligned}\right.
\end{align*}
We take the initial conditions as follows:
\begin{align*}
(hu)(x,0)=0,\quad h(x,0) =
\left\{\begin{aligned}
& 1-b(x)+\epsilon & \text{ if~~} 1.1\leq x\leq 1.2,\\
& 1-b(x) & \text{ otherwise, }
\end{aligned}\right.
\end{align*}
where $\epsilon$ is a non-zero perturbation constant.
\end{exmp}

In Example \ref{examp3}, two perturbations $\epsilon=0.2$ (big pulse) and $\epsilon =0.001$ (small pulse)
are used to test the scheme. For small $\epsilon$, this disturbance should generate two waves, propagating
opposite directions at the characteristic speeds $\pm \sqrt{gh}$. It is difficult to involving such
small perturbations of the water surface for many
numerical methods \cite{LeVeque1998}. We use $200$ uniform cells with simple transmissive boundary
condition to compute the solution at time $t=0.2$. Figures \ref{fig_examp3_big} and
\ref{fig_examp3_small} show the surface level and discharge for the big and small pulse cases respectively.
For comparison, we also show the $2000$ cells solution as reference. We can see that our scheme
successfully capture these waves on the relative coarse mesh and there are no obvious spurious numerical oscillations.
\begin{figure}[htb]
\centering
\caption{ \label{fig_examp3_big} Example \ref{examp3}: $\epsilon=0.2$.}
\subfigure[Surface level]{
\includegraphics[width=0.45\textwidth]{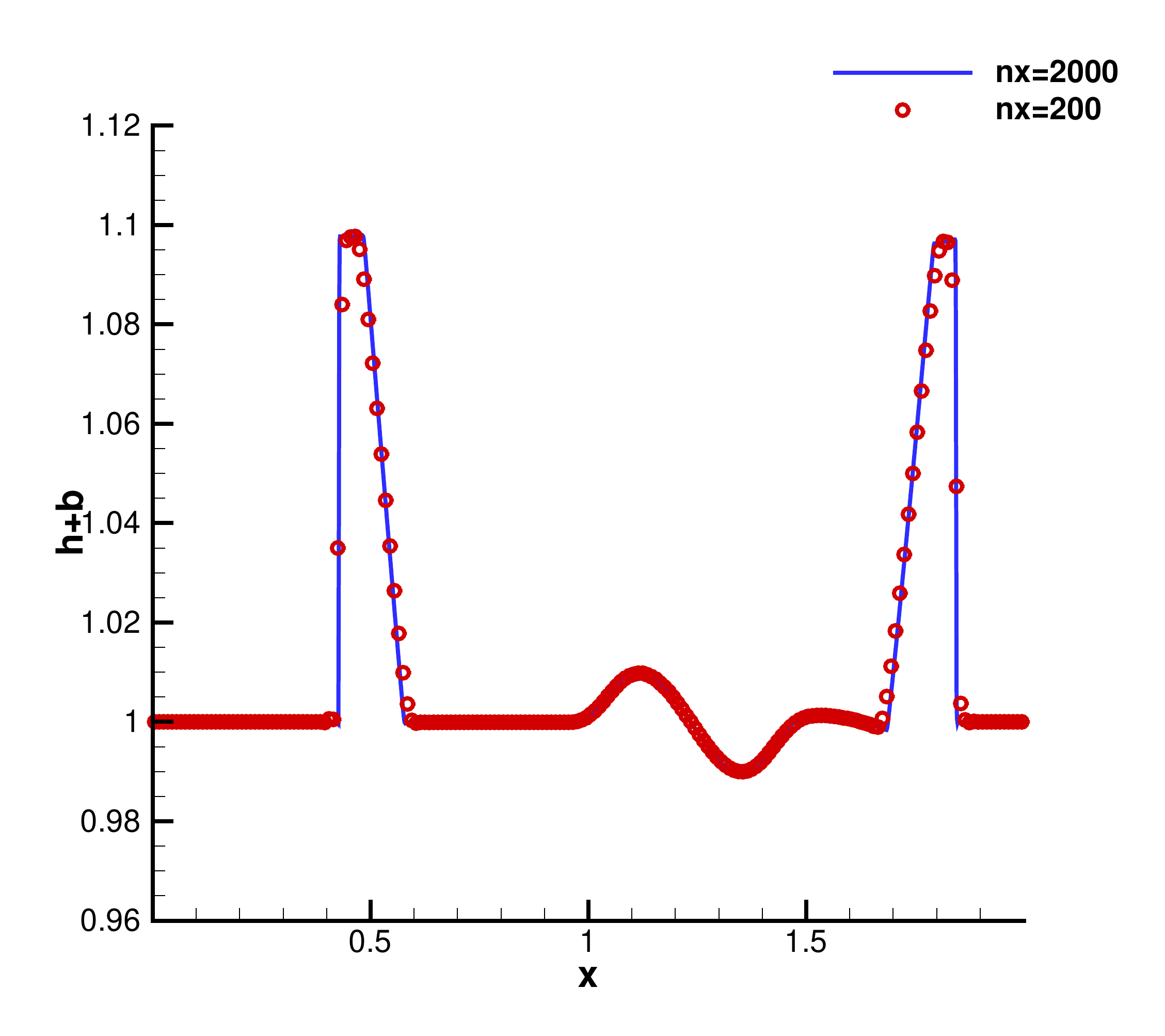} }
\subfigure[The discharge]{
\includegraphics[width=0.45\textwidth]{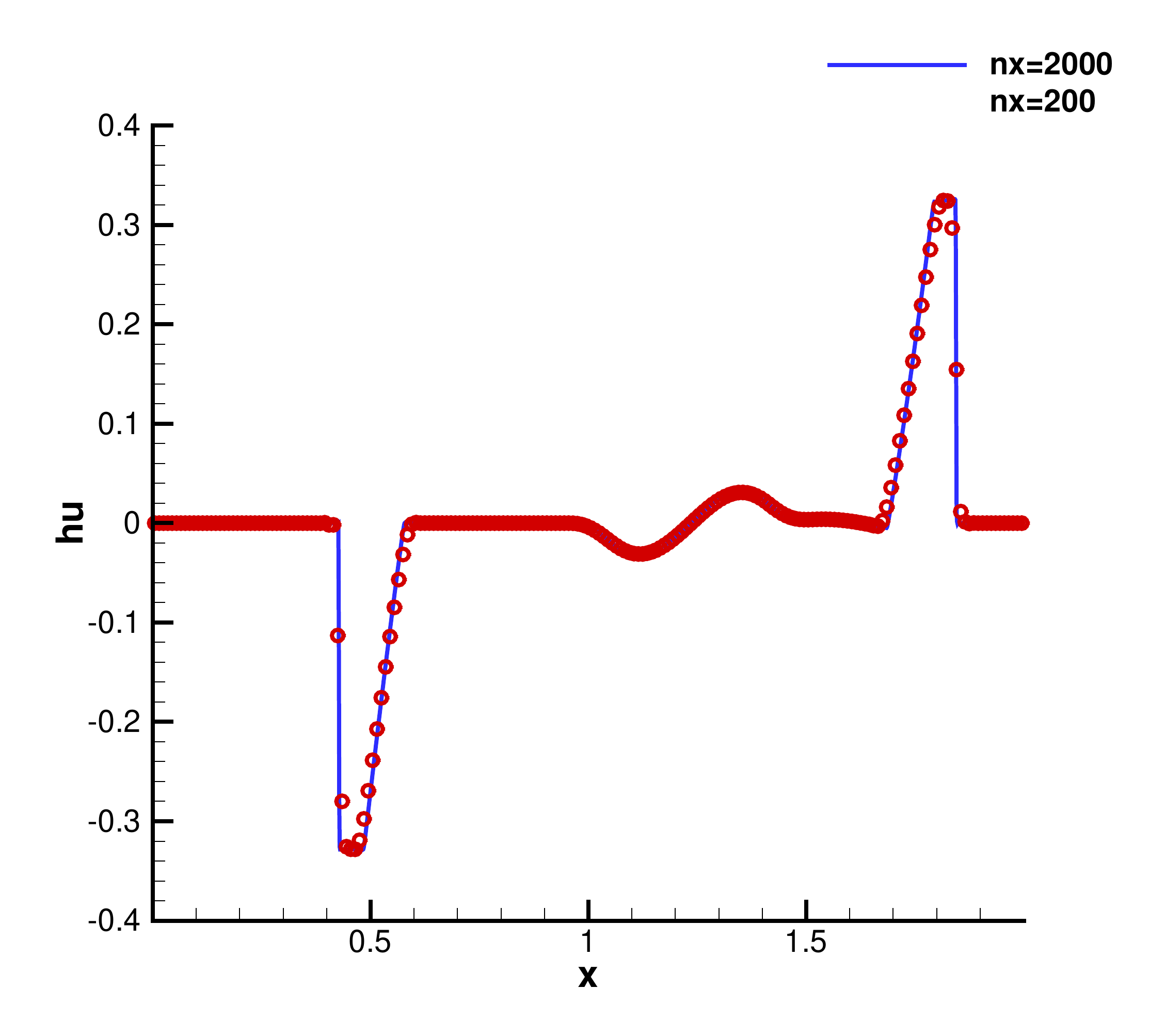} }
\end{figure}

\begin{figure}[htb]
\centering
\caption{ \label{fig_examp3_small} Example \ref{examp3}: $\epsilon=0.001$.}
\subfigure[Surface level]{
\includegraphics[width=0.45\textwidth]{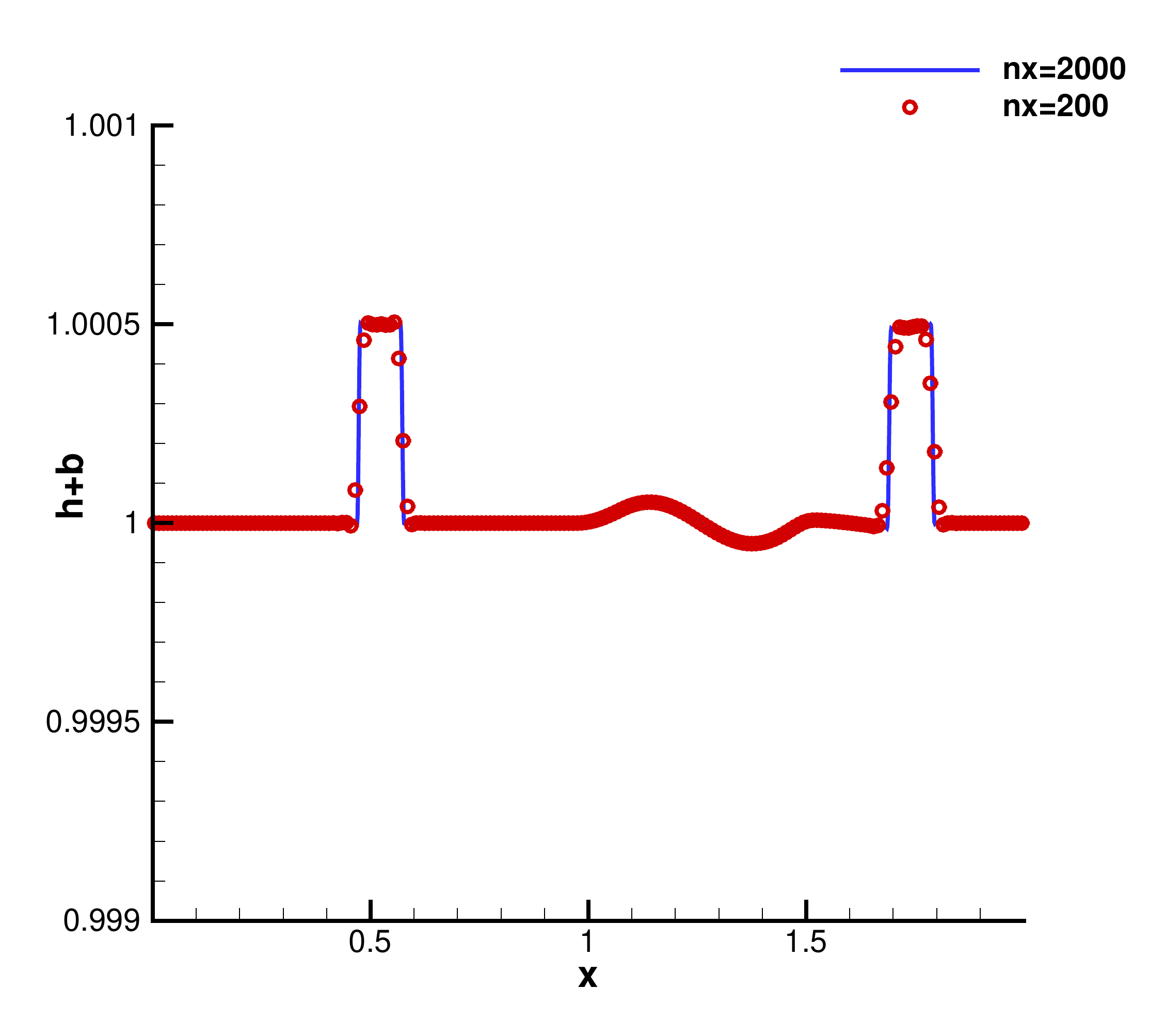} }
\subfigure[The discharge]{
\includegraphics[width=0.45\textwidth]{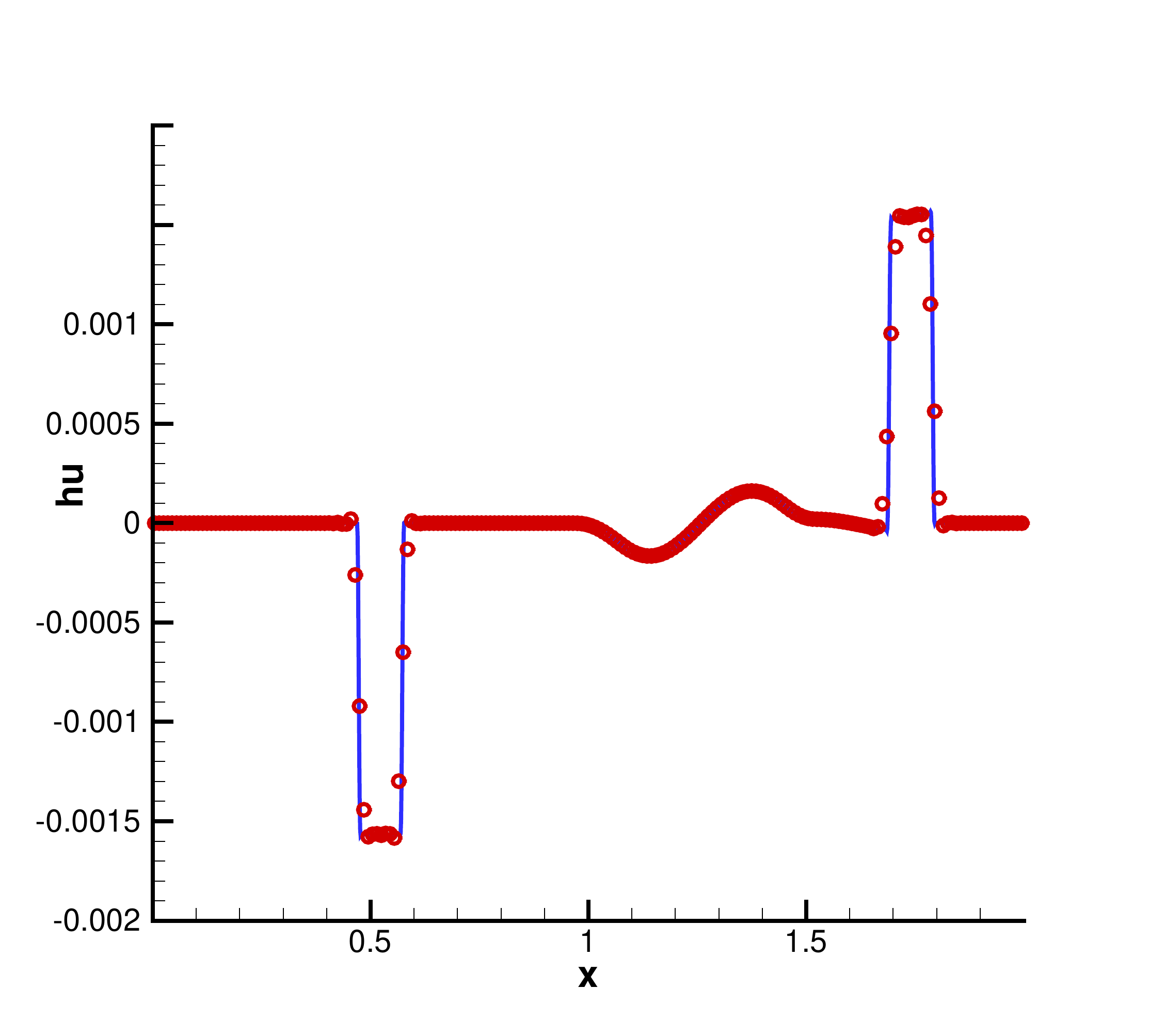} }
\end{figure}

\begin{exmp}\label{examp4}
Next, we simulate the dam breaking problem over a rectangular bump \cite{XSCICP2006}. We use this example to to test the scheme in the case of a rapidly varying flow over discontinuous bottom topography.
The bed level is given with
\begin{align*}
b(x)=\left\{\begin{aligned}
8& \text{ if~~ } |x-750|\leq 187.5,\\
0& \text{ otherwise, }
\end{aligned}\right.\quad x\in [0,1500].
\end{align*}
The initial data are taken as follows
\begin{align*}
(hu)(x,0)=0,\quad h(x,0)=\left\{\begin{array}{cc}
20-b(x)& \text{ if } x \leq 750,\\
15-b(x) & \text{ otherwise.}
\end{array}\right.
\end{align*}
\end{exmp}
We compute the solutions at $t=15$ and $t=60$ by using $400$ uniform cells and use the results of $4000$ uniform cells as reference solutions. Figure \ref{fig_examp4_1} shows the numerical
results have good resolution and non-oscillatory which agree well with the reference solution.
\begin{figure}[htb]
\centering
\caption{ \label{fig_examp4_1} The dam breaking problem}
\subfigure[$t=15$]{
\includegraphics[width=0.45\textwidth]{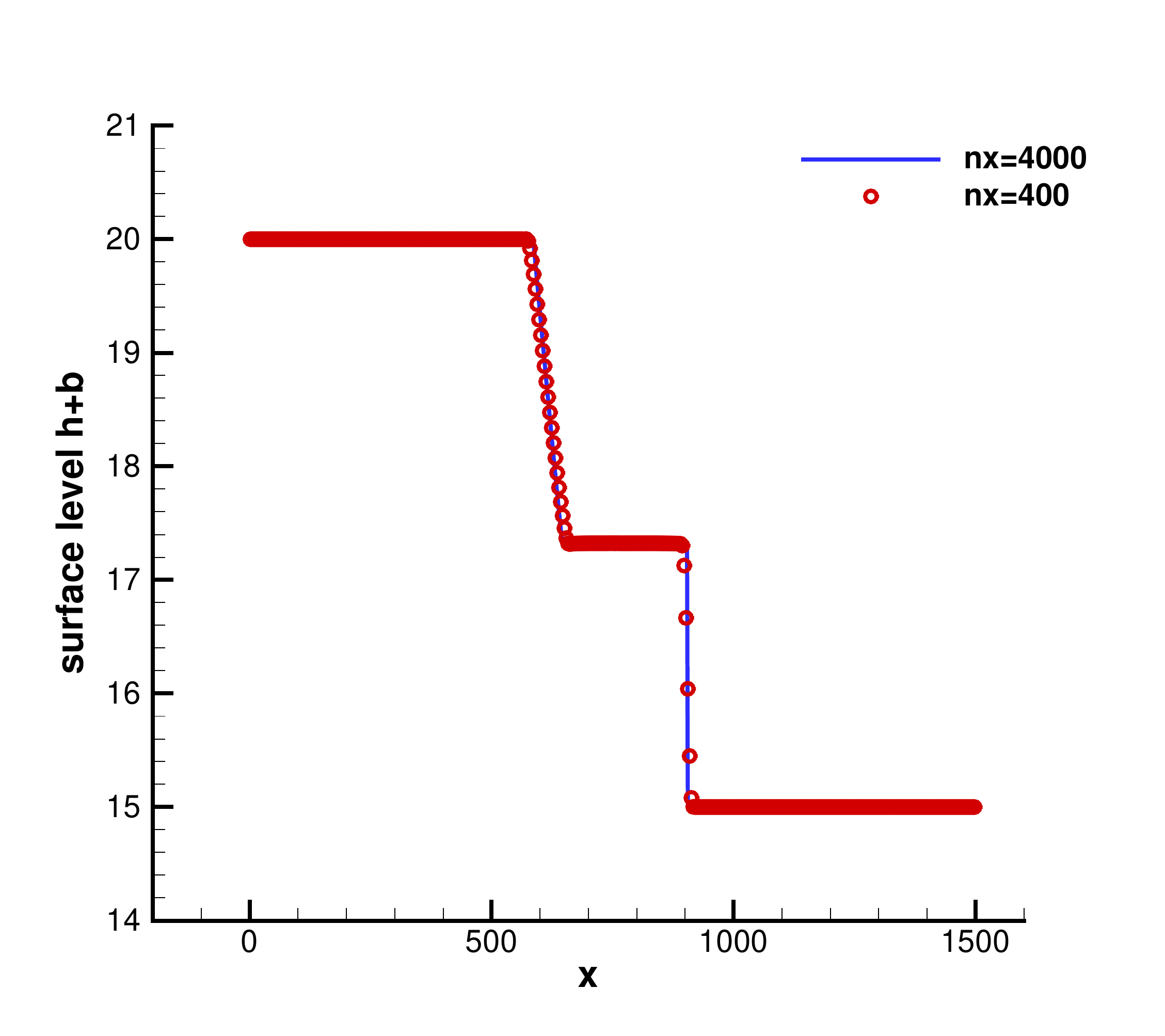} }
\subfigure[$t=60$]{
\includegraphics[width=0.45\textwidth]{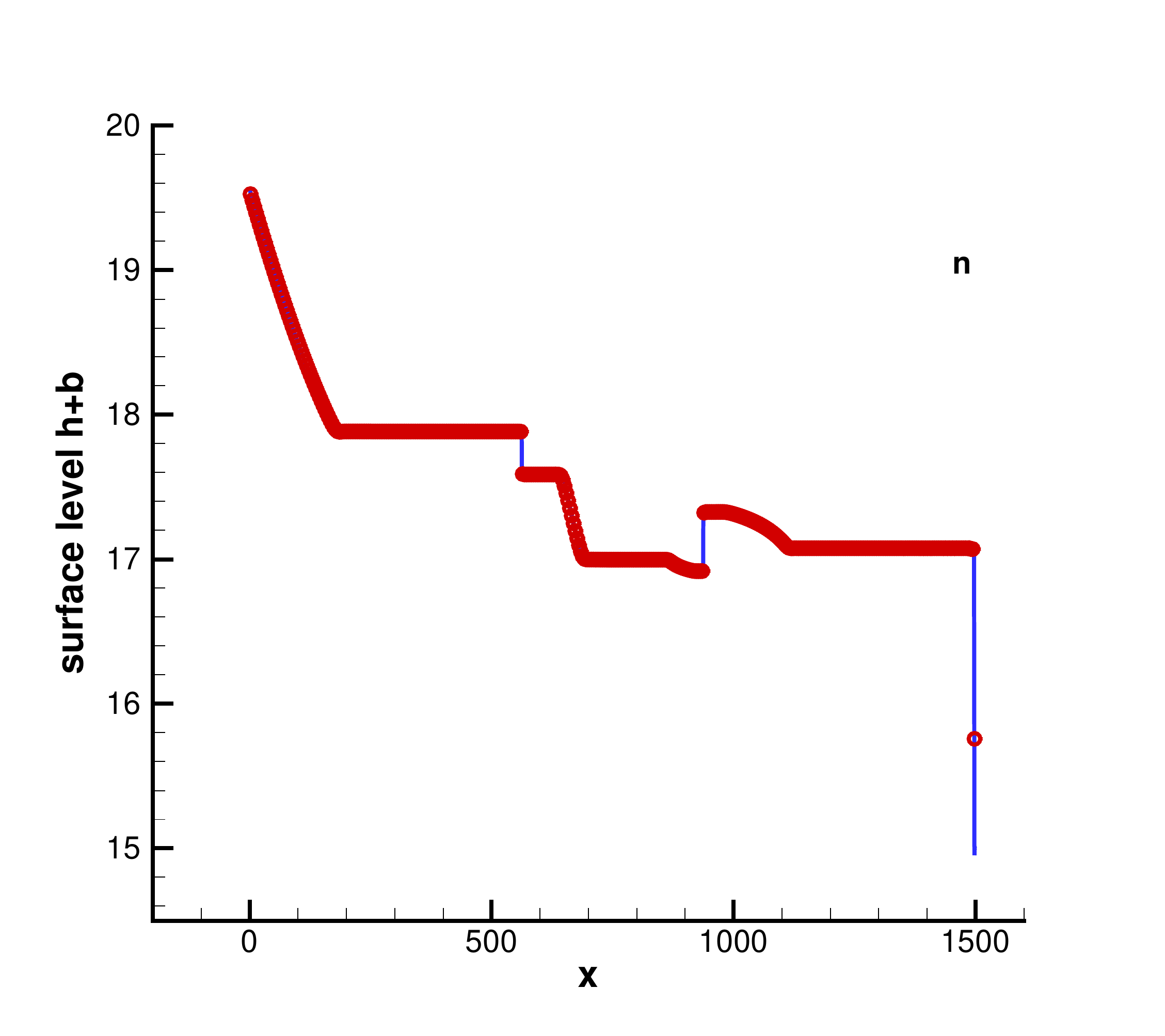} }
\end{figure}

\begin{exmp}\label{examp5}
The last one-dimensional example is to study the convergence in time towards the steady flow
over a bump. They are widely
used to test numerical schemes for shallow water equations \cite{Goutal1997,Vazeuez1999,XSCICP2006}.
We take the bottom function as:
\begin{align*}
b(x)=\left\{\begin{array}{ll}
0.2-0.05(x-10)^2 & \text{ if } 8 \leq x \leq 12,\\
\quad\quad 0 & \text{ otherwise },
\end{array}\right.
\end{align*}
The length of channel is $[0,25]$ and the initial conditions are given by
\begin{align*}
h(x,0)=0.5-b(x),\quad u(x,0)=0.
\end{align*}
\end{exmp}
Theoretically, the flow
can be subcritical or transcritical with or without a steady shock on different boundary conditions.
We compute the solution until $t=200s$ on $200$ uniform meshes. We consider three types of
boundary conditions:

a) Transcritical flow without a shock: The boundary condition is $hu=1.53m^2/s$ at $x=0$ and
$h=0.66m$ at $x=25$.

We plot the surface level $h+b$ and discharge $hu$ in Figure \ref{fig_examp5_1}. The numerical results
 show very good agreement with the analytical solution \cite{Goutal1997}.

b) Transcritical flow with a shock: The boundary conditions are $hu=0.18m^2/s$ at $x=0$
and $h=0.33m$ at $x=25$.

Here we also plot the surface level $h+b$ and the discharge $hu$ in Figure \ref{fig_examp5_2}. We can
observe that the stationary shock appears on the surface. However, there is one overshoot near the jump
on the discharge $hu$. Since our scheme is the still water well-balanced scheme not moving water
well-balanced scheme, the similar result can also be found in \cite{Li2017JSC}.

c) Subcritical flow: The boundary conditions are $hu=4.42m^2/s$ at $x=0$ and $h=2m$ at $x=25$.

We also plot the surface level $h+b$ and the discharge $hu$ in Figure \ref{fig_examp5_3},
and observe that the results have good performances comparing with the analytical solutions \cite{Goutal1997}.

\begin{figure}[htb]
\centering
\caption{ \label{fig_examp5_1} The transcritical flow without a shock}
\subfigure[]{
\includegraphics[width=0.45\textwidth]{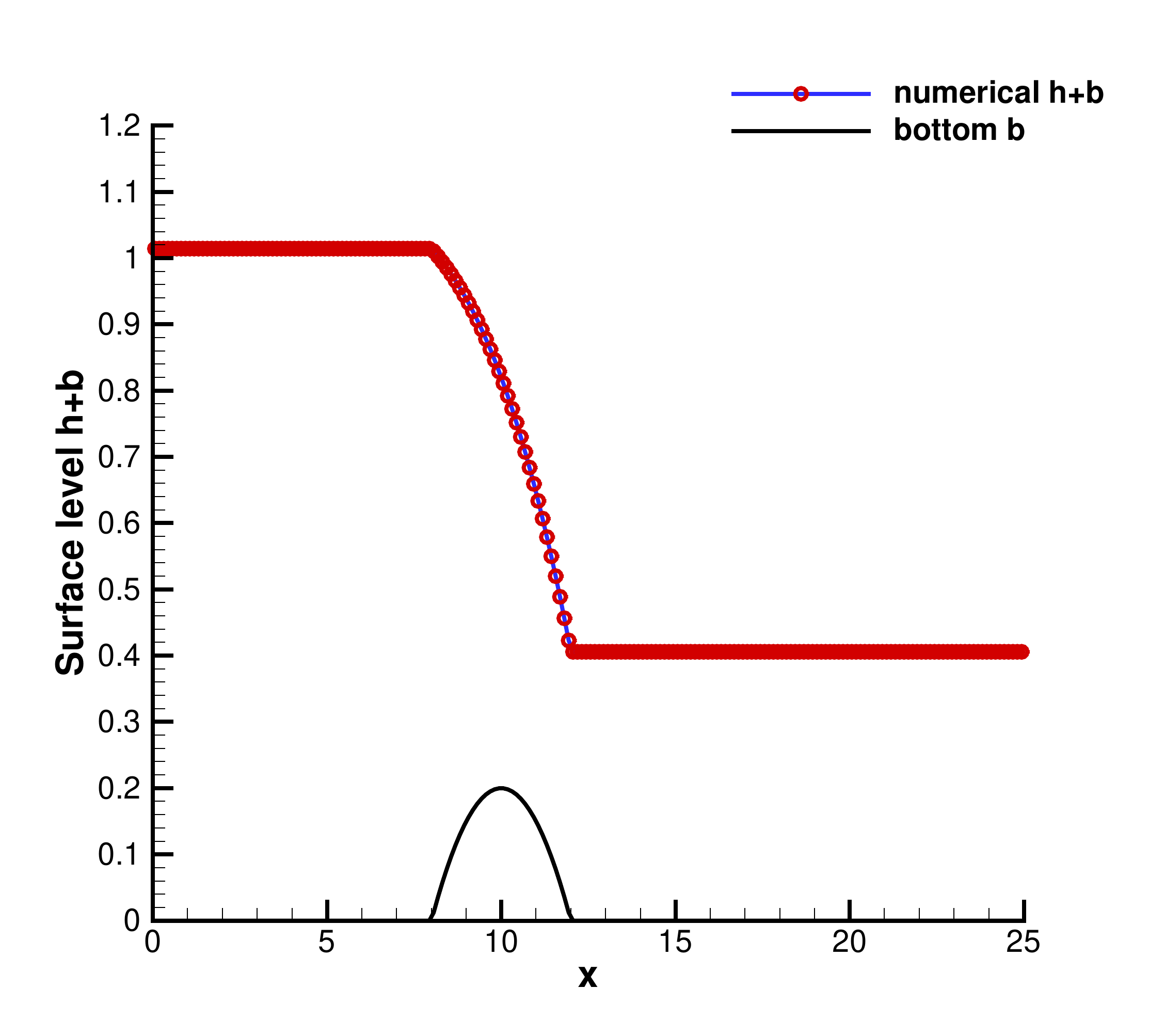} }
\subfigure[]{
\includegraphics[width=0.45\textwidth]{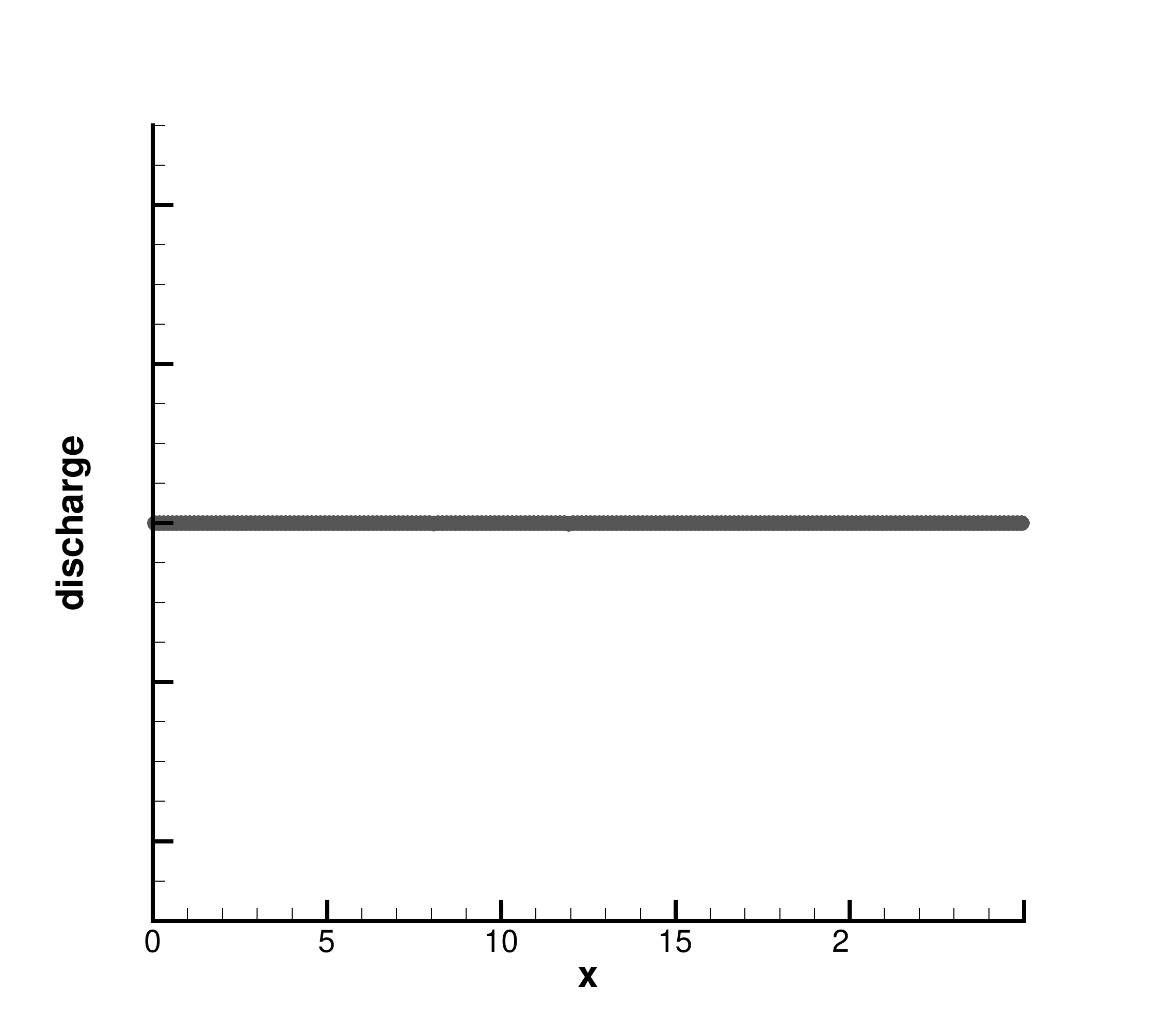} }
\end{figure}

\begin{figure}[htb]
\centering
\caption{ \label{fig_examp5_2} The transcritical flow with a shock}
\subfigure[]{
\includegraphics[width=0.45\textwidth]{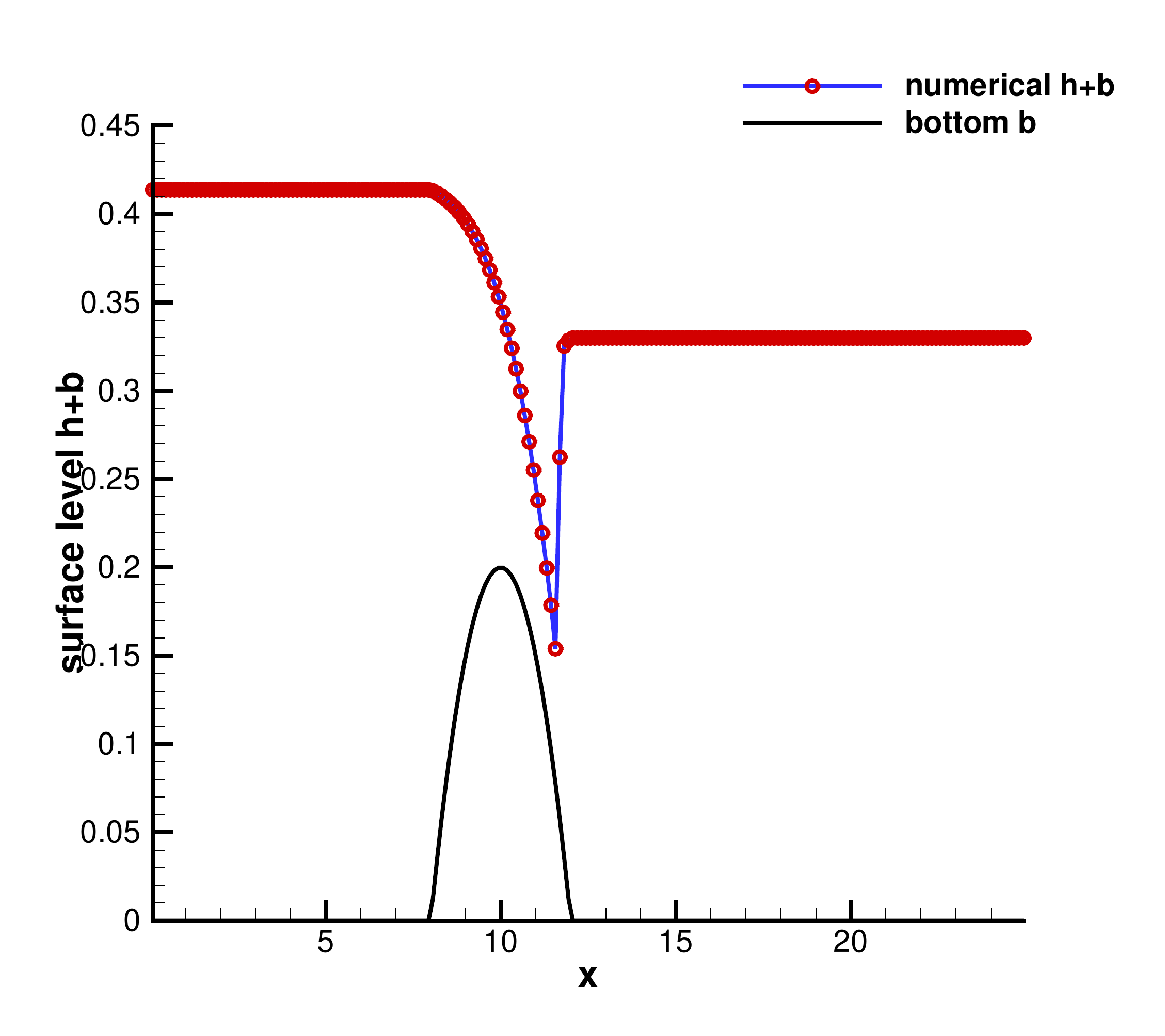} }
\subfigure[]{
\includegraphics[width=0.45\textwidth]{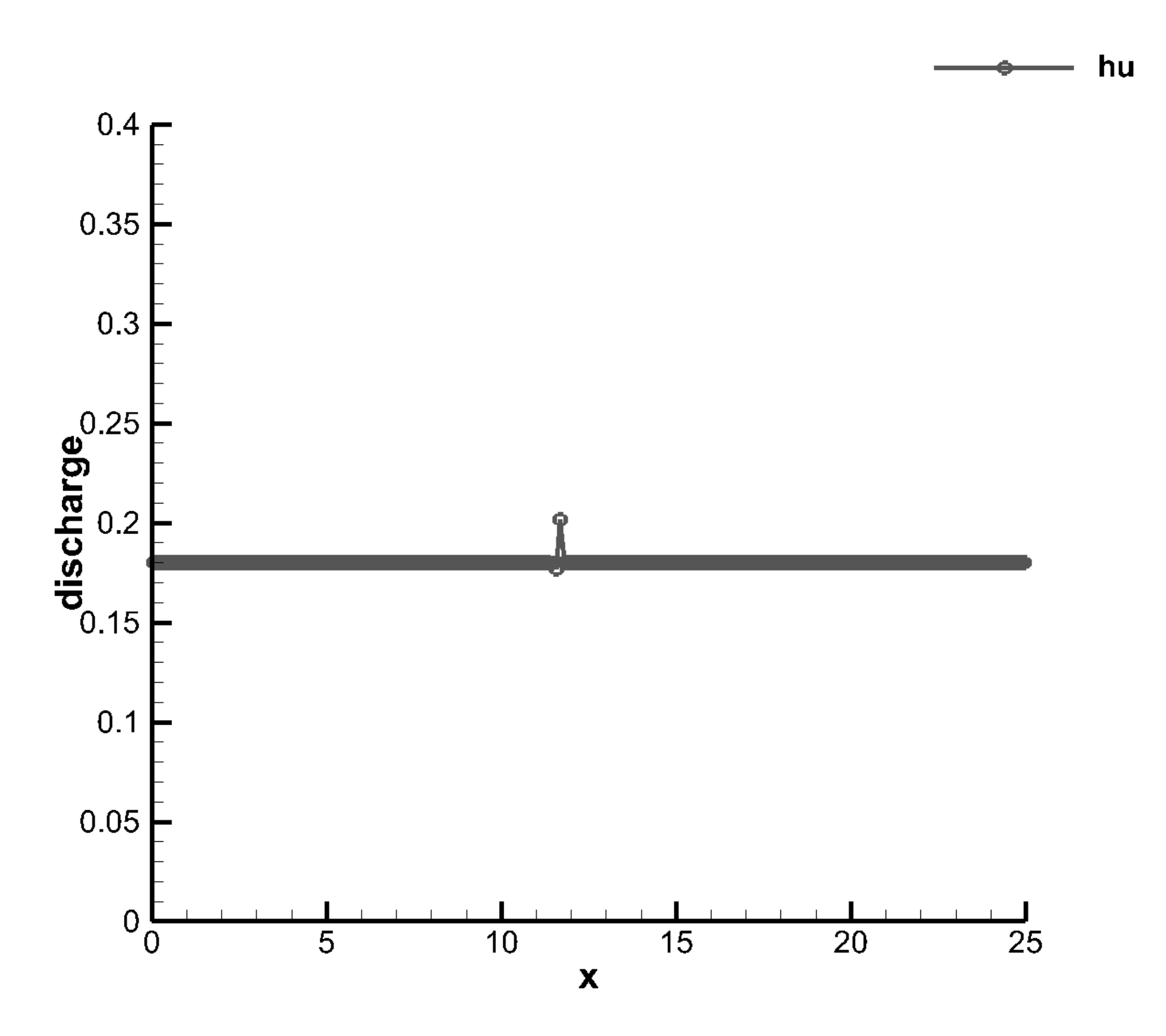} }
\end{figure}

\begin{figure}[htb]
\centering
\caption{ \label{fig_examp5_3} The steady subcritical flow}
\subfigure[]{
\includegraphics[width=0.45\textwidth]{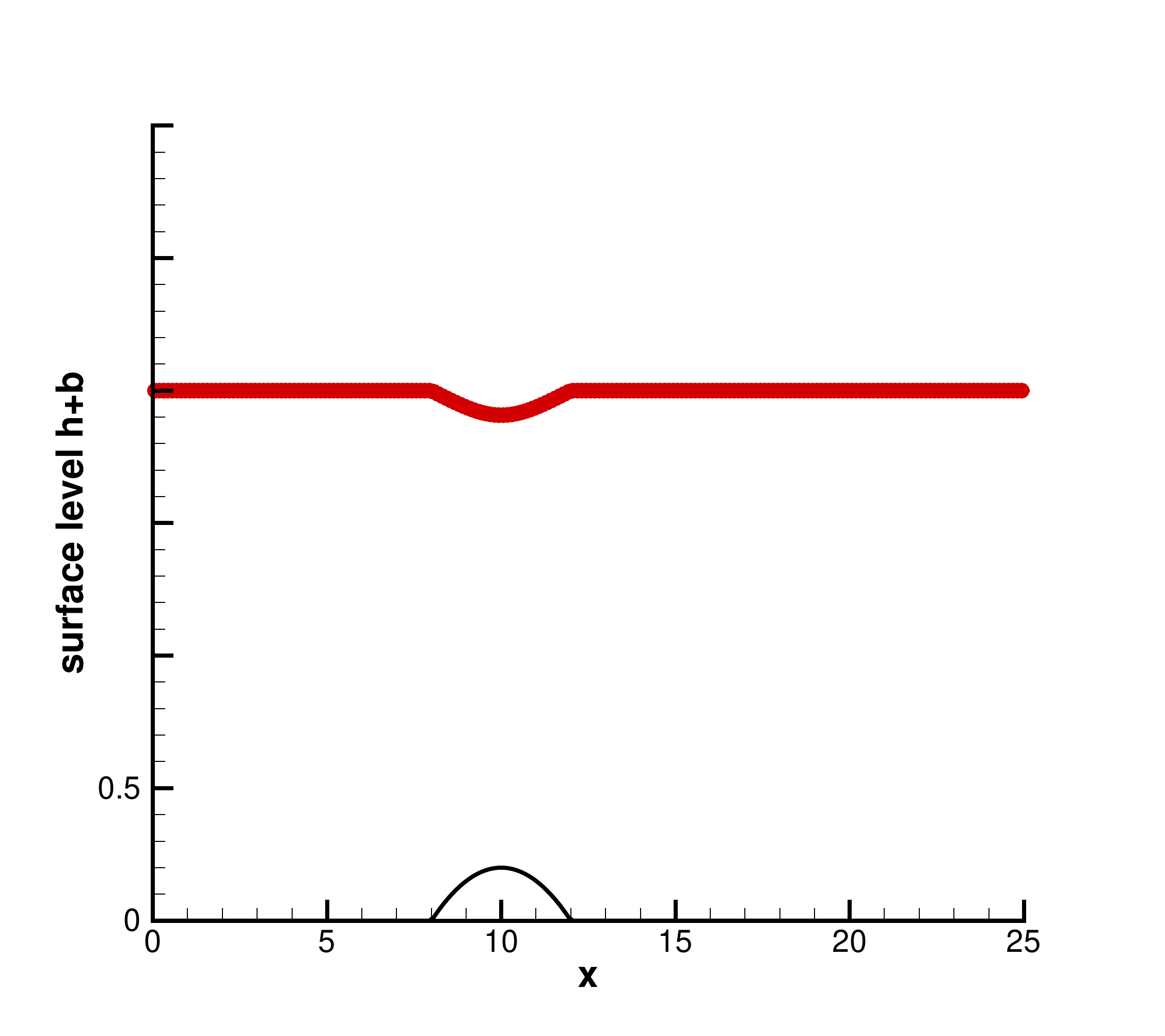} }
\subfigure[]{
\includegraphics[width=0.45\textwidth]{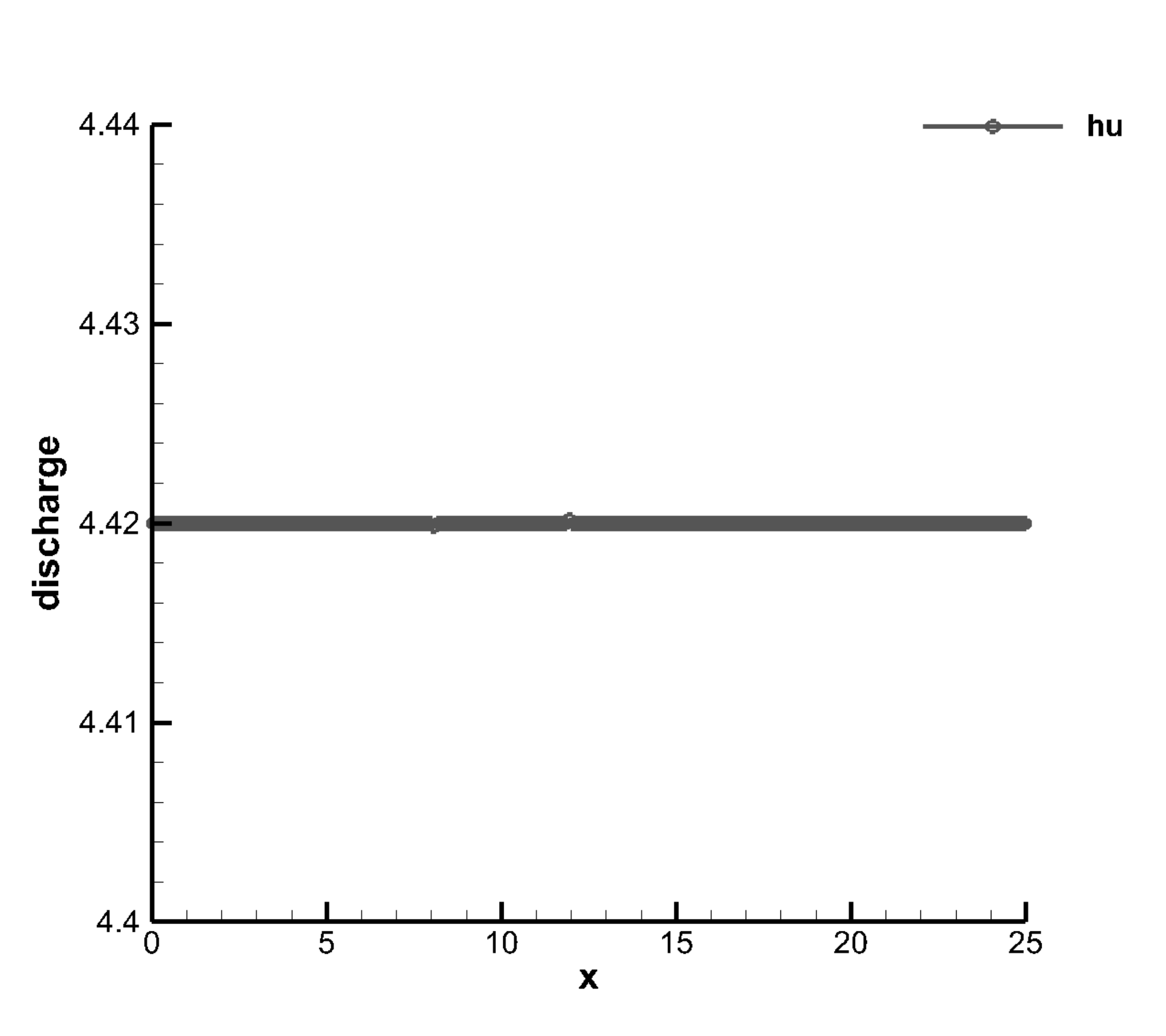} }
\end{figure}

\subsection{Two-dimensional Problems}
\label{sec_numeric_2d}

\begin{exmp}\label{examp6}
Now let us test the well-balanced property in two dimensions. We consider a non-flat bottom as follows:
\begin{align*}
b(x,y) = 0.8e^{-50((x-0.5)^2+(y-0.5)^2)}, \quad (x,y) \in [0,1]^2.
\end{align*}
The initial stationary state is given by:
$$h(x,y) = 1 - b(x,y),\quad  v(x,y,0)=u(x,y,0)=0.$$
\end{exmp}
We use the single, double and quadruple precisions to compute the numerical solutions up to time
$t=0.1$ on the $100\times 100$ uniform mesh. It can be found that the scheme preserve the
still water state from Table \ref{tab_2d_1}. For all numerical solutions, the $L^2$ and $L^\infty$ errors achieve the machine accuracy for different precisions, which confirms that the proposed scheme does preserve  the well-balanced property in two-dimensions.

\begin{table}[htb]
\centering
\caption{$L^2$ and $L^\infty$ errors for the stationary solution with different precisions in Example \ref{examp6}. \label{tab_2d_1}}
\resizebox{\textwidth}{!}{
\begin{tabular}{|c|ccc|ccc|}
  \hline
              & \multicolumn{3}{c|}{$L^2$ error} & \multicolumn{3}{c|}{$L^\infty$ error}\\ \cline{2-7}
    precision & $h$ & $hu$ & $hv$ & $h$& $hu$ & $hv$ \\ \hline
    single        &  4.034E-06    &  4.778E-06    &  4.823E-06    &  5.364E-06    &  4.227E-05    &  4.221E-05     \\ \hline
    double    &  1.892E-14    &  1.055E-14    &  1.041E-14    &  2.143E-14    &  7.965E-14    &  7.678E-14        \\ \hline
    quadruple &  1.303E-32    &  9.898E-33    &  1.076E-32    &  1.483E-32    &  5.068E-32    &  4.904E-32
   \\\hline
\end{tabular}
}
\end{table}

\begin{exmp}\label{examp7}
Next, we test the numerical accuracy of our schemes. Let us consider the following smooth bottom
function and initial conditions:
\begin{align*}
& b(x,y)=\sin(2\pi x)+\cos(2\pi y), \quad h(x,y,0)=10+e^{\sin(2\pi x)}\cos(2 \pi y),\\
& (hu)(x,y,0) = \sin(\cos(2\pi x))\sin(2\pi y), \quad (hv)(x,y,0)=\cos(2\pi x)\cos(\sin(2\pi y))\, ,
\end{align*}
where $(x,y)\in\Omega = [0,1]^2$ and the periodic boundary conditions are imposed.
\end{exmp}

We compute the solution at time $t=0.05$ before the shock appears in the solution.
We still use the a posteriori error $\|\bfu_h-\bfu_{\frac h2}\|$ as numerical errors of $\bfu_h$.
From Tables \ref{tab_2d_SW_smooth_h}-\ref{tab_2d_SW_smooth_hv}, we can observe the
optimal convergence rate for all numerical solutions.

\begin{table}[htb]\small
\caption{\label{tab_2d_SW_smooth_h} $h$'s numerical errors and orders
in Example \ref{examp7}. } \centering
\medskip
\begin{tabular}{|c|r||cc|cc|cc|}  \hline
 & $N\times N$ & $L^1$ error & order & $L^2$ error & order & $L^\infty$ error & order \\ \hline \hline
 \multirow{6}{0.6cm}{$P^1$}
    &   $10\times10$    &  9.202E-02    &    --    &  1.174E-01    &    --   &  4.798E-01    &    --   \\
    &   $20\times20$    &  2.110E-02    &   2.125    &  3.002E-02    &   1.968    &  1.783E-01    &   1.428   \\
    &   $40\times40$    &  4.318E-03    &   2.289    &  6.515E-03    &   2.204    &  5.325E-02    &   1.744   \\
    &   $80\times80$    &  9.109E-04    &   2.245    &  1.366E-03    &   2.254    &  1.292E-02    &   2.044   \\
    &  $160\times160$    &  2.120E-04    &   2.103    &  3.158E-04    &   2.113    &  2.879E-03    &   2.165   \\
    &  $320\times320$    &  5.167E-05    &   2.037    &  7.717E-05    &   2.033    &  7.884E-04    &   1.869   \\\hline
 \multirow{6}{0.6cm}{$P^2$}
    &   $10\times10$    &  1.543E-02    &    --    &  2.472E-02    &    --    &  1.416E-01    &    --   \\
    &   $20\times20$    &  1.949E-03    &   2.985    &  4.203E-03    &   2.556    &  3.601E-02    &   1.975   \\
    &   $40\times40$    &  2.056E-04    &   3.245    &  4.995E-04    &   3.073    &  8.813E-03    &   2.031   \\
    &   $80\times80$    &  2.366E-05    &   3.119    &  6.080E-05    &   3.038    &  1.521E-03    &   2.534   \\
    &  $160\times160$    &  2.837E-06    &   3.060    &  7.543E-06    &   3.011    &  2.381E-04    &   2.676   \\
    &  $320\times320$    &  3.506E-07    &   3.017    &  9.414E-07    &   3.002    &  3.221E-05    &   2.886   \\
\hline
 \multirow{6}{0.6cm}{$P^3$}
    &   $10\times10$    &  3.722E-03    &    --    &  7.671E-03    &    --    &  4.978E-02    &    --   \\
    &   $20\times20$    &  2.475E-04    &   3.910    &  7.182E-04    &   3.417    &  1.087E-02    &   2.196   \\
    &   $40\times40$    &  1.496E-05    &   4.049    &  4.786E-05    &   3.908    &  1.544E-03    &   2.815   \\
    &   $80\times80$    &  8.752E-07    &   4.095    &  3.412E-06    &   3.810    &  1.477E-04    &   3.386   \\
    &  $160\times160$    &  5.075E-08    &   4.108    &  2.041E-07    &   4.063    &  9.506E-06    &   3.957   \\
    &  $320\times320$   &  3.087E-09    &   4.039    &  1.260E-08    &   4.018    &  5.972E-07    &   3.993   \\ \hline
\end{tabular}
\end{table}

\begin{table}[htb]\small
\caption{\label{tab_2d_SW_smooth_hu} $hu$'s numerical errors and orders
in Example \ref{examp7}. } \centering
\medskip
\begin{tabular}{|c|r||cc|cc|cc|}  \hline
 & $N$ & $L^1$ error & order & $L^2$ error & order & $L^\infty$ error & order \\ \hline \hline
 \multirow{6}{0.6cm}{$P^1$}
    &   $10\times10$    &  4.287E-01    &    --    &  5.806E-01    &    --    &  2.048E+00    &    --   \\
    &   $20\times20$    &  9.343E-02    &   2.198    &  1.315E-01    &   2.143    &  5.354E-01    &   1.935   \\
    &   $40\times40$    &  1.684E-02    &   2.472    &  2.318E-02    &   2.504    &  1.102E-01    &   2.281   \\
    &   $80\times80$    &  3.208E-03    &   2.392    &  4.235E-03    &   2.453    &  2.275E-02    &   2.276   \\
    &  $160\times160$    &  7.060E-04    &   2.184    &  9.118E-04    &   2.215    &  5.696E-03    &   1.998   \\
    &  $320\times320$    &  1.683E-04    &   2.068    &  2.176E-04    &   2.067    &  1.434E-03    &   1.990   \\\hline
 \multirow{6}{0.6cm}{$P^2$}
    &   $10\times10$    &  5.715E-02    &    --    &  8.461E-02    &    --    &  3.783E-01    &    --   \\
    &   $20\times20$   &  6.255E-03    &   3.192    &  9.655E-03    &   3.132    &  7.481E-02    &   2.338   \\
    &   $40\times40$    &  7.633E-04    &   3.035    &  1.191E-03    &   3.020    &  1.481E-02    &   2.336   \\
    &  $80\times80$    &  1.269E-04    &   2.588    &  1.982E-04    &   2.587    &  2.379E-03    &   2.638   \\
    &  $160\times160$    &  2.259E-05    &   2.490    &  3.646E-05    &   2.443    &  3.577E-04    &   2.734   \\
    &  $320\times320$    &  3.849E-06    &   2.553    &  6.413E-06    &   2.507    &  4.974E-05    &   2.846   \\\hline
 \multirow{6}{0.6cm}{$P^3$}
    &   $10\times10$    &  1.333E-02    &    --    &  2.119E-02    &    --    &  1.246E-01    &    --   \\
    &   $20\times20$    &  1.027E-03    &   3.698    &  1.751E-03    &   3.597    &  2.477E-02    &   2.331   \\
    &   $40\times40$    &  7.112E-05    &   3.852    &  1.252E-04    &   3.806    &  2.645E-03    &   3.227   \\
    &   $80\times80$    &  4.532E-06    &   3.972    &  8.284E-06    &   3.917    &  1.969E-04    &   3.748   \\
    &  $160\times160$    &  2.703E-07    &   4.067    &  5.206E-07    &   3.992    &  1.513E-05    &   3.702   \\
    &  $320\times320$    &  1.477E-08    &   4.194    &  3.045E-08    &   4.096    &  1.035E-06    &   3.870   \\\hline
\end{tabular}
\end{table}

\begin{table}[htb]\small
\caption{\label{tab_2d_SW_smooth_hv} $hv$'s numerical errors and orders
in Example \ref{examp7}. } \centering
\medskip
\begin{tabular}{|c|r||cc|cc|cc|}  \hline
 & $N$ & $L^1$ error & order & $L^2$ error & order & $L^\infty$ error & order \\ \hline \hline
 \multirow{6}{0.6cm}{$P^1$}
    &   $10\times10$    &  6.891E-01    &    --    &  9.023E-01    &    --    &  2.335E+00    &    --   \\
    &   $20\times20$    &  1.596E-01    &   2.110    &  2.392E-01    &   1.916    &  1.043E+00    &   1.163   \\
    &   $40\times40$    &  3.281E-02    &   2.282    &  5.205E-02    &   2.200    &  3.485E-01    &   1.581   \\
    &   $80\times80$    &  6.891E-03    &   2.251    &  1.091E-02    &   2.254    &  9.071E-02    &   1.942   \\
    &  $160\times160$    &  1.611E-03    &   2.097    &  2.535E-03    &   2.106    &  2.115E-02    &   2.100   \\
    &  $320\times320$    &  3.923E-04    &   2.038    &  6.213E-04    &   2.029    &  5.564E-03    &   1.927   \\\hline
 \multirow{6}{0.6cm}{$P^2$}
    &   $10\times10$    &  1.302E-01    &    --    &  2.193E-01    &    --    &  1.370E+00    &    --   \\
    &   $20\times20$    &  1.683E-02    &   2.952    &  3.687E-02    &   2.572    &  3.430E-01    &   1.998   \\
    &   $40\times40$    &  1.778E-03    &   3.243    &  4.268E-03    &   3.111    &  7.771E-02    &   2.142   \\
    &   $80\times80$    &  2.072E-04    &   3.101    &  5.186E-04    &   3.041    &  1.333E-02    &   2.544   \\
    &  $160\times160$    &  2.558E-05    &   3.018    &  6.440E-05    &   3.009    &  1.977E-03    &   2.753   \\
    &  $320\times320$    &  3.324E-06    &   2.944    &  8.107E-06    &   2.990    &  2.590E-04    &   2.932   \\ \hline
 \multirow{6}{0.6cm}{$P^3$}
     &   $10\times10$    &  3.380E-02    &    --    &  7.089E-02    &    --    &  5.048E-01    &    --   \\
    &   $20\times20$    &  2.265E-03    &   3.900    &  6.032E-03    &   3.555    &  1.060E-01    &   2.252   \\
    &   $40\times40$    &  1.501E-04    &   3.915    &  4.106E-04    &   3.877    &  1.369E-02    &   2.953   \\
    &   $80\times80$    &  9.436E-06    &   3.992    &  2.992E-05    &   3.778    &  1.212E-03    &   3.498   \\
    &  $160\times160$    &  5.554E-07    &   4.087    &  1.794E-06    &   4.060    &  7.502E-05    &   4.014   \\
    &  $320\times320$    &  3.287E-08    &   4.078    &  1.113E-07    &   4.010    &  4.624E-06    &   4.020   \\ \hline
\end{tabular}
\end{table}

\begin{exmp}\label{examp8}
In the last example, we test the ability of the proposed scheme to capture the perturbation of the still water equilibrium in 2D. This example is widely used to test the well-balanced schemes, which is given by LeVeque \cite{LeVeque1998}. The bottom function and the initial data are given by:
\begin{align*}
b(x,y)=0.8e^{-5(x-0.9)^2-50(y-0.5)^2}.
\end{align*}
\begin{align*}
\begin{aligned}
& h(x,y,0)=\left\{\begin{array}{ll}
1-b(x,y) + 0.01 & \text{ if ~~} 0.05 \leq x \leq 0.15\\
1-b(x,y) & \text{ otherwise }
\end{array}\right., \\
&  (hu)(x,y,0) = (hv)(x,y,0) = 0  \, ,
\end{aligned}
\end{align*}
where $(x,y)\in [0,2]\times[0,1]$.
\end{exmp}
We compute the numerical solutions on $200\times 100$ and $600\times 300$ uniform meshes respectively.
Numerical solutions at different time are presented for comparison in Figure \ref{fig_examp8}. These results demonstrate that our
schemes can resolve the complex small features without obvious spurious oscillations.

\begin{figure}[htb]
\centering
\caption{ \label{fig_examp8} The contours of the surface level $h+b$ for Example \ref{examp8}.
Left: results with a $200\times 100$ uniform mesh. Right: results with a $600\times 300$ uniform mesh.}
\subfigure{
\includegraphics[width=0.45\textwidth]{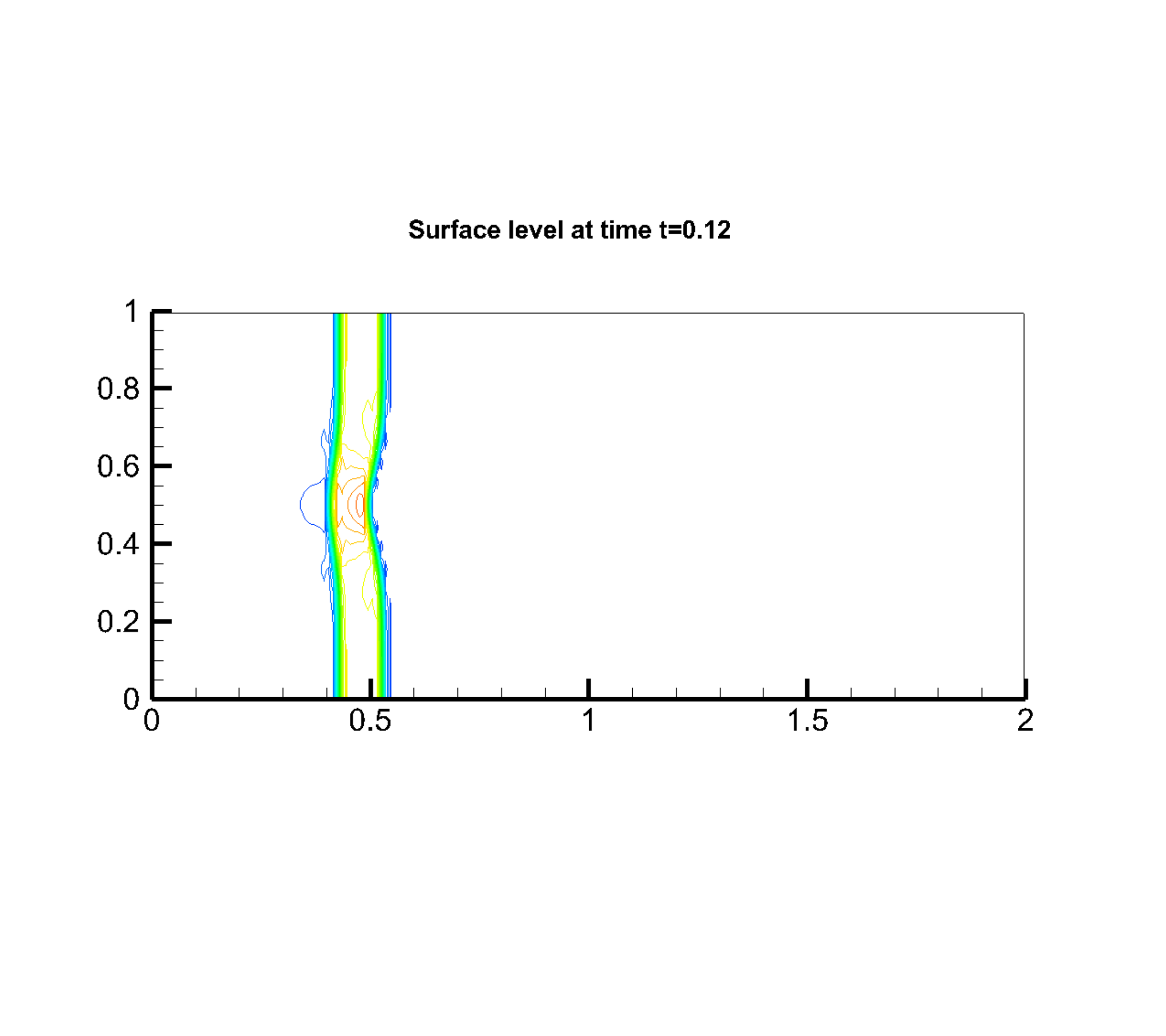} }
\subfigure{
\includegraphics[width=0.45\textwidth]{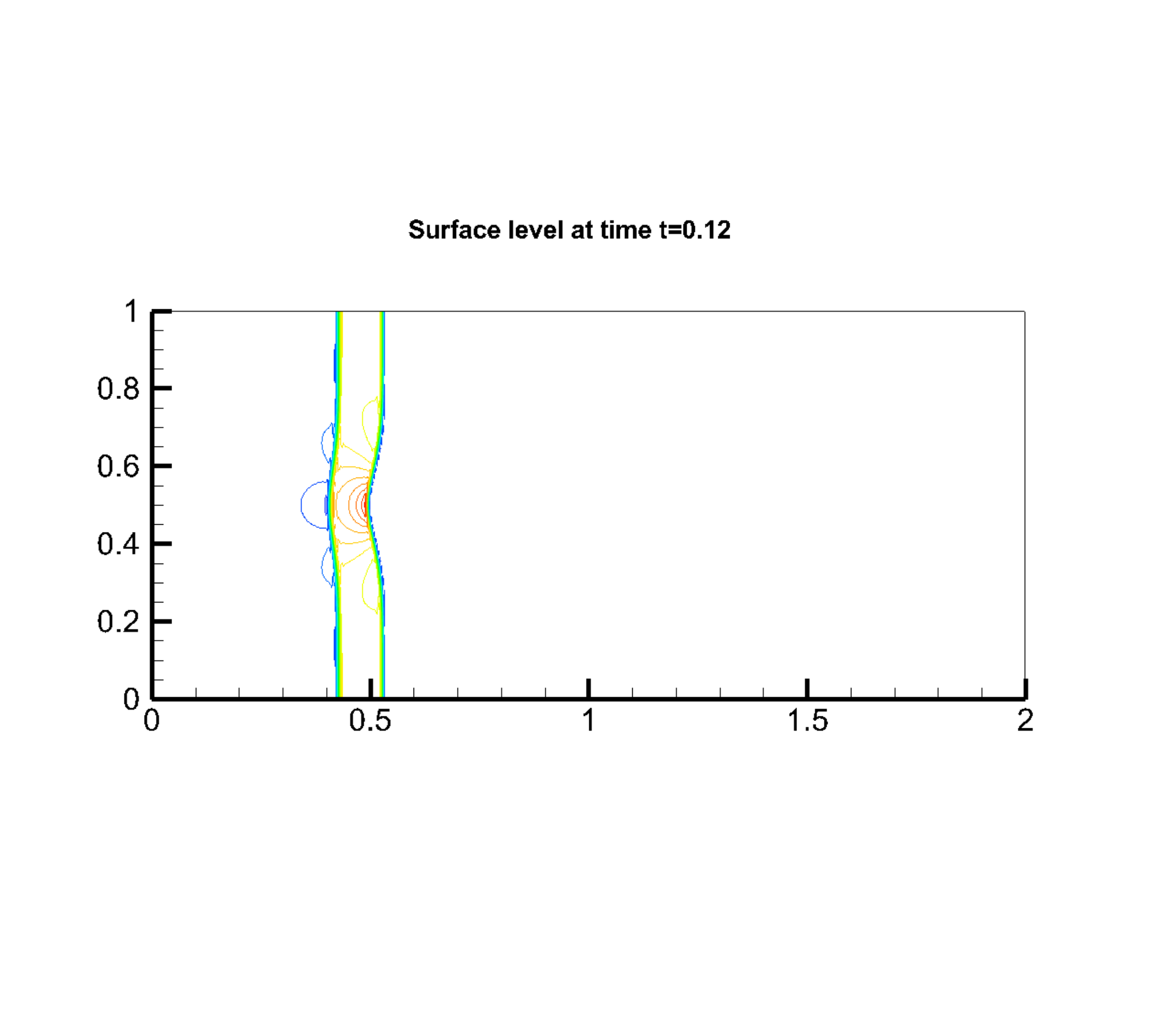} }\\
\vspace{-2.0cm}
\subfigure{
\includegraphics[width=0.45\textwidth]{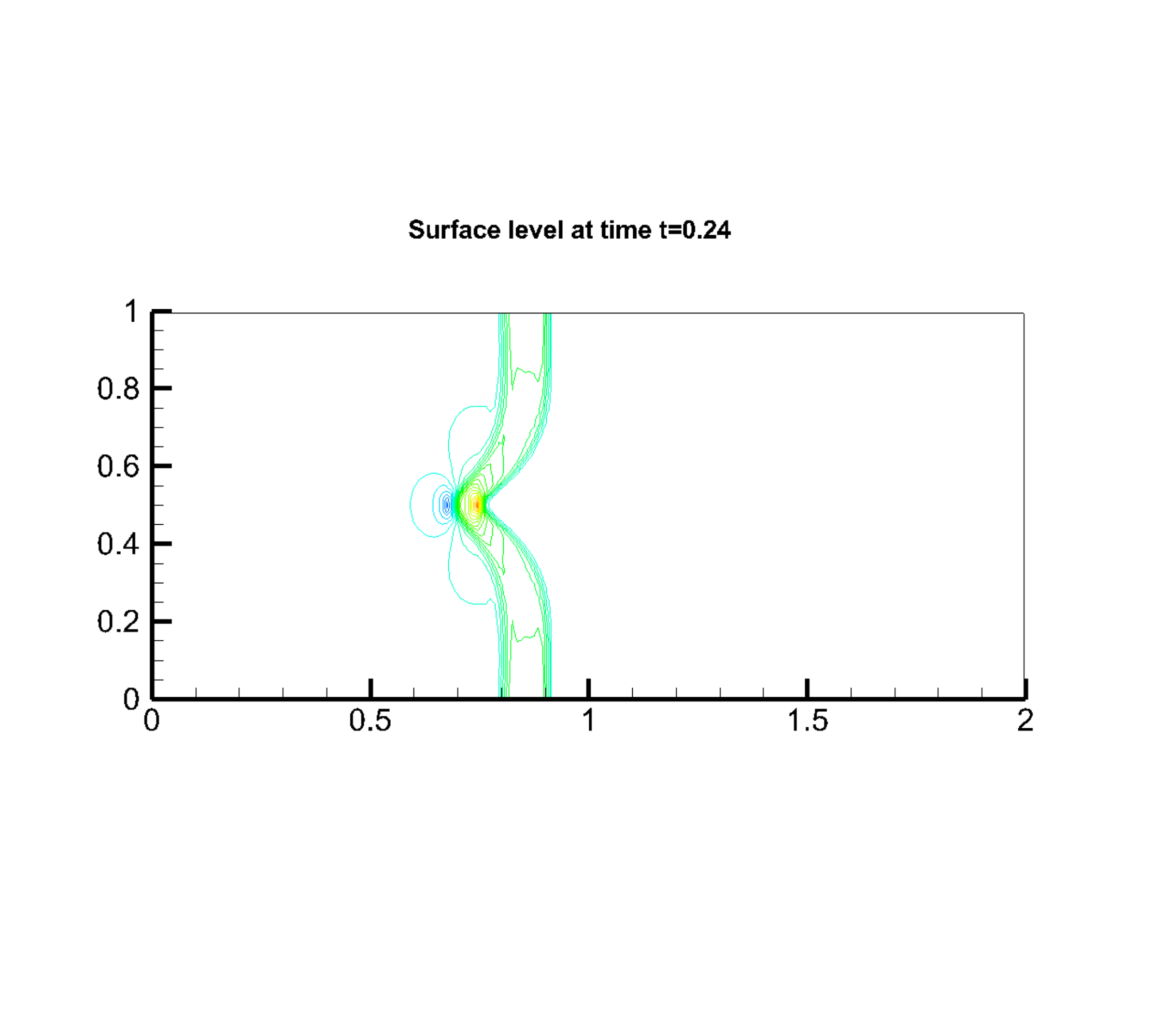} }
\subfigure{
\includegraphics[width=0.45\textwidth]{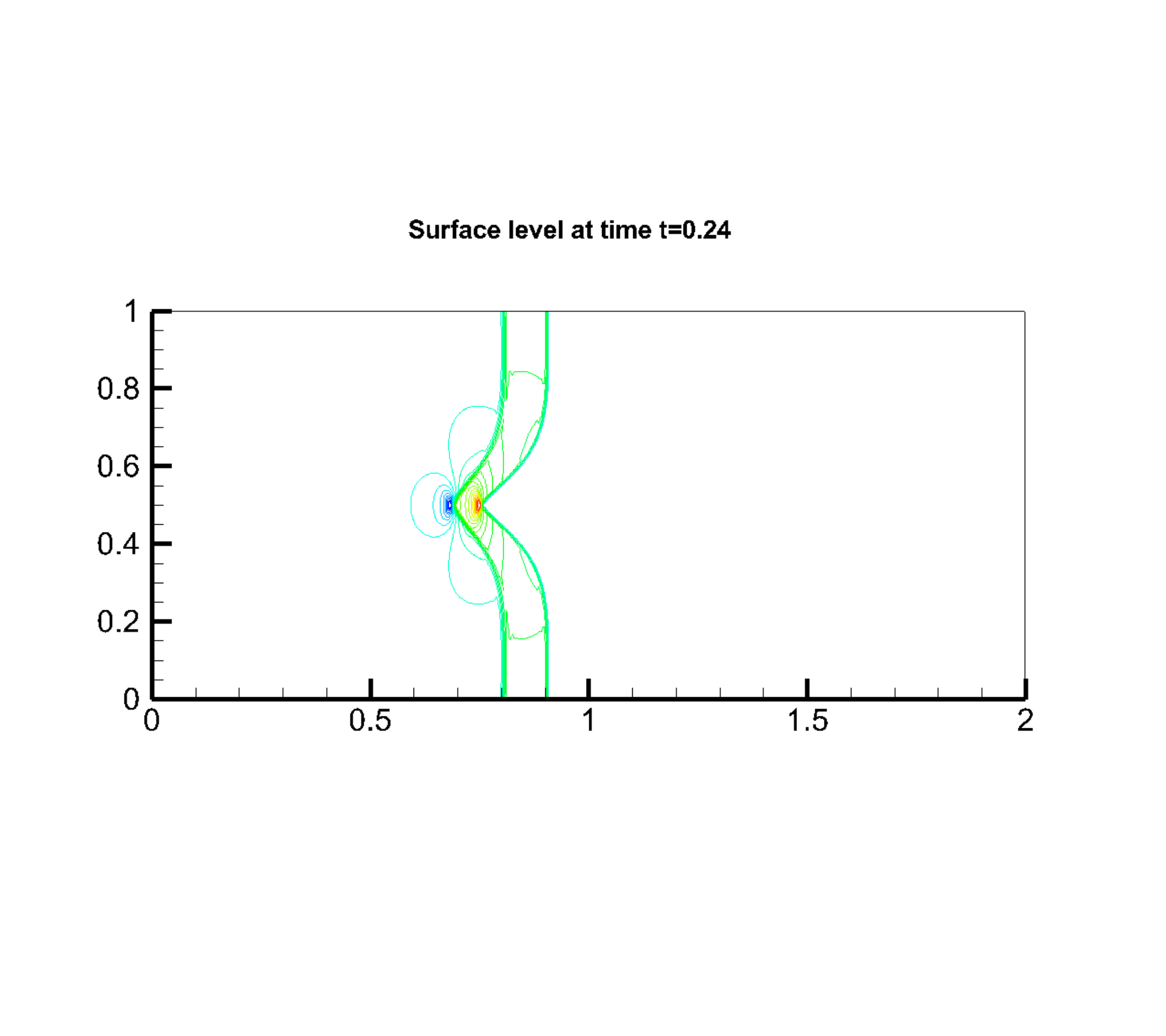} }\\
\vspace{-2.0cm}
\subfigure{
\includegraphics[width=0.45\textwidth]{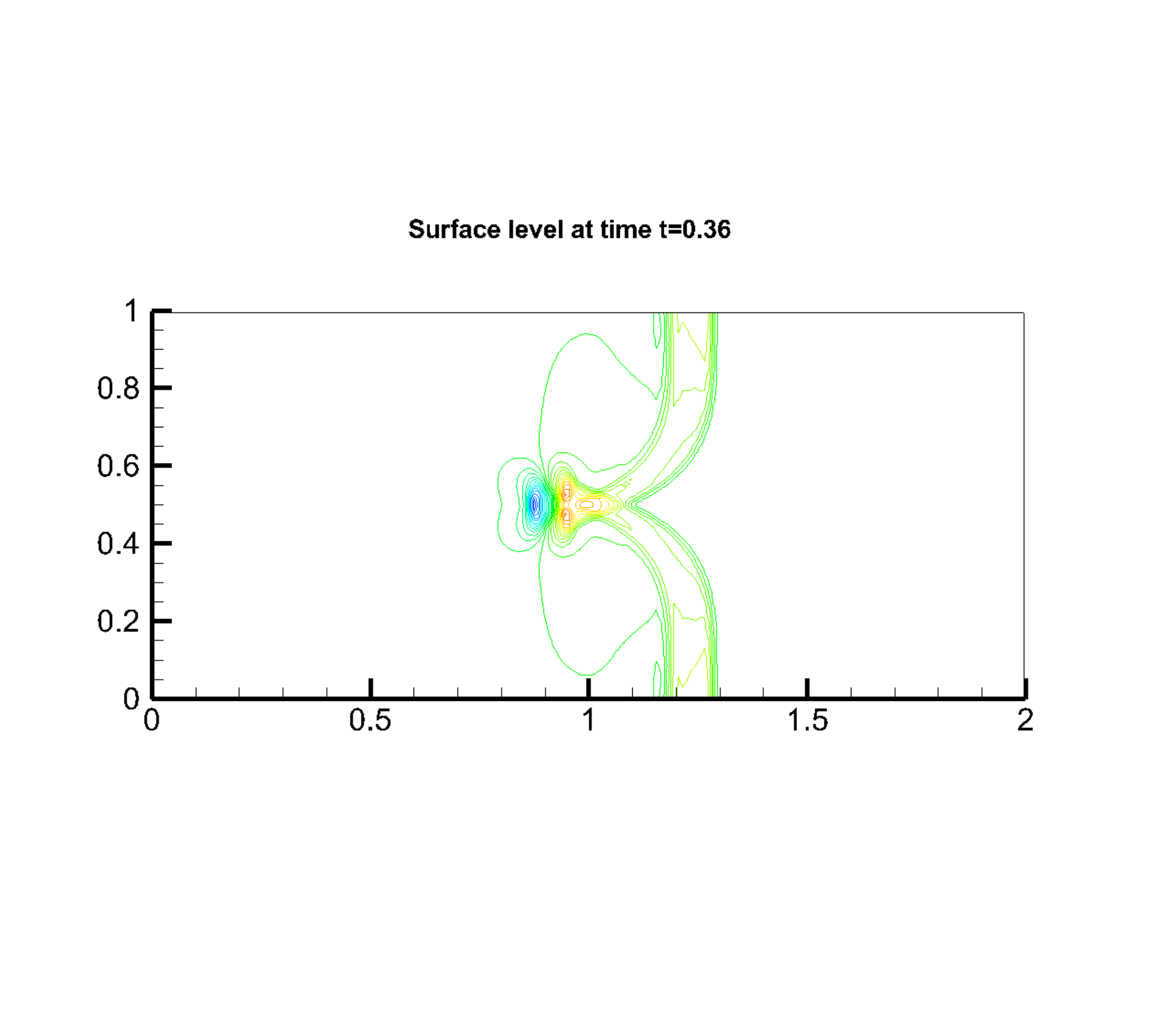} }
\subfigure{
\includegraphics[width=0.45\textwidth]{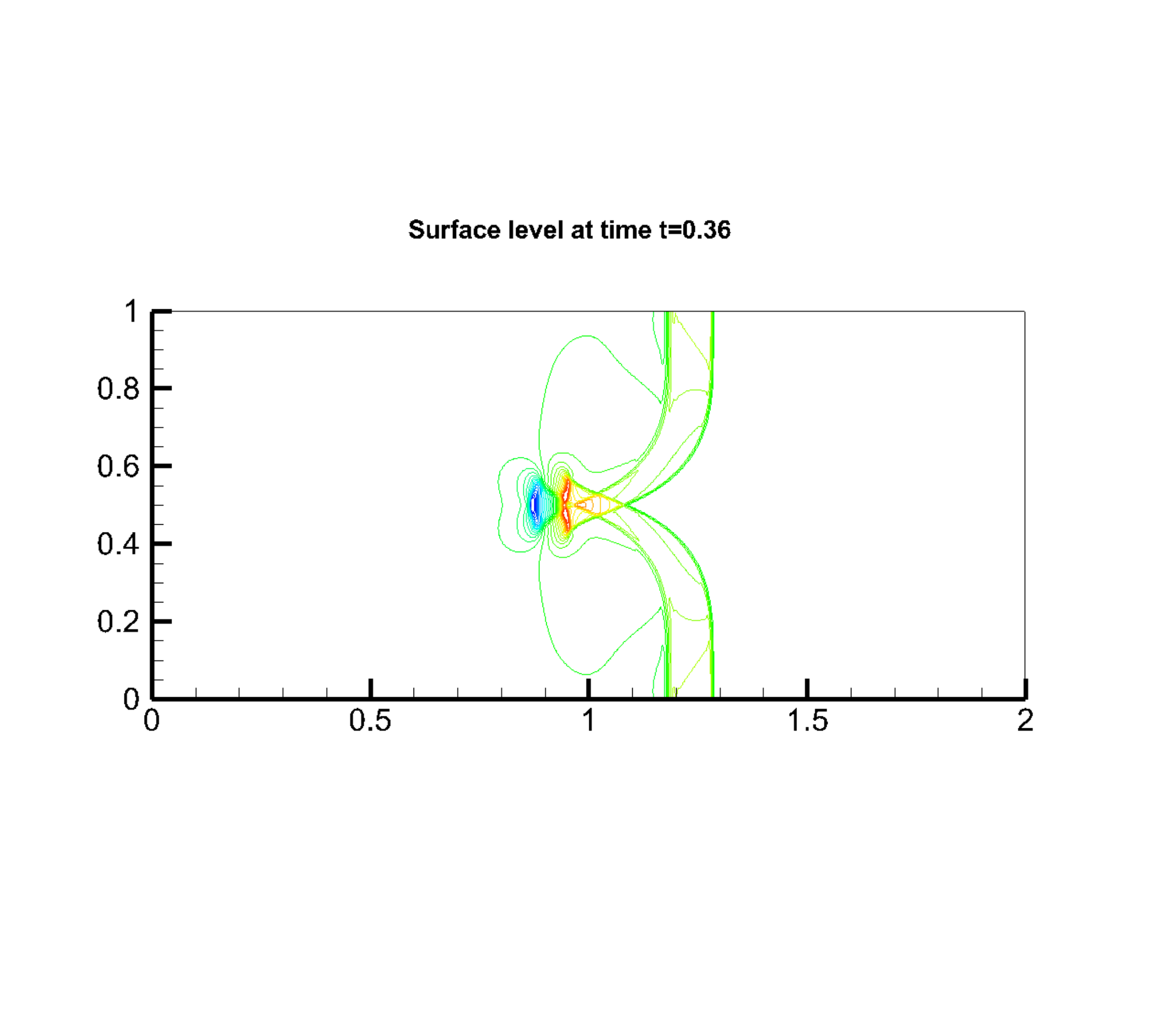} }\\
\vspace{-2.0cm}
\subfigure{
\includegraphics[width=0.45\textwidth]{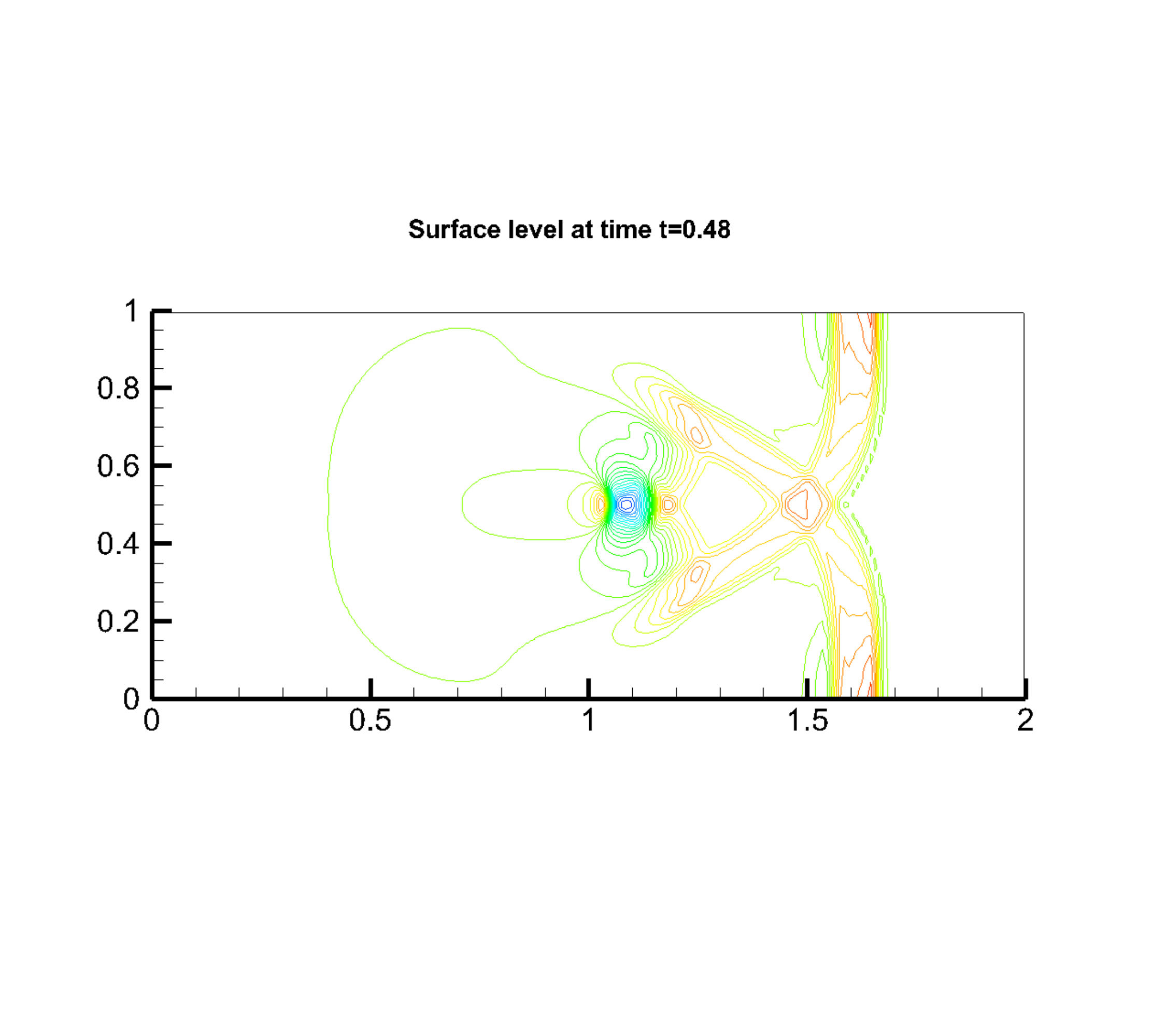} }
\subfigure{
\includegraphics[width=0.45\textwidth]{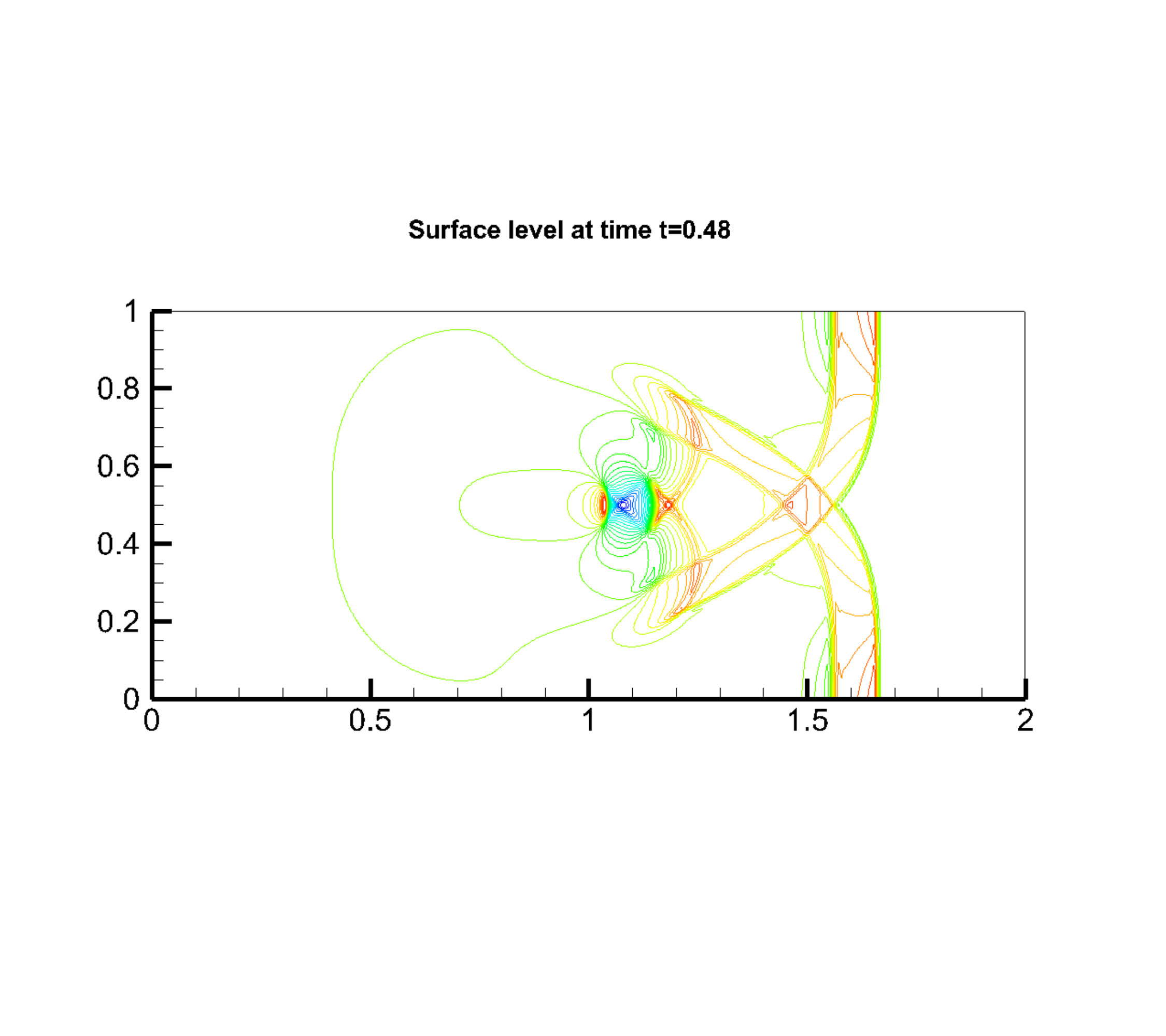} }\\
\vspace{-2.0cm}
\subfigure{
\includegraphics[width=0.45\textwidth]{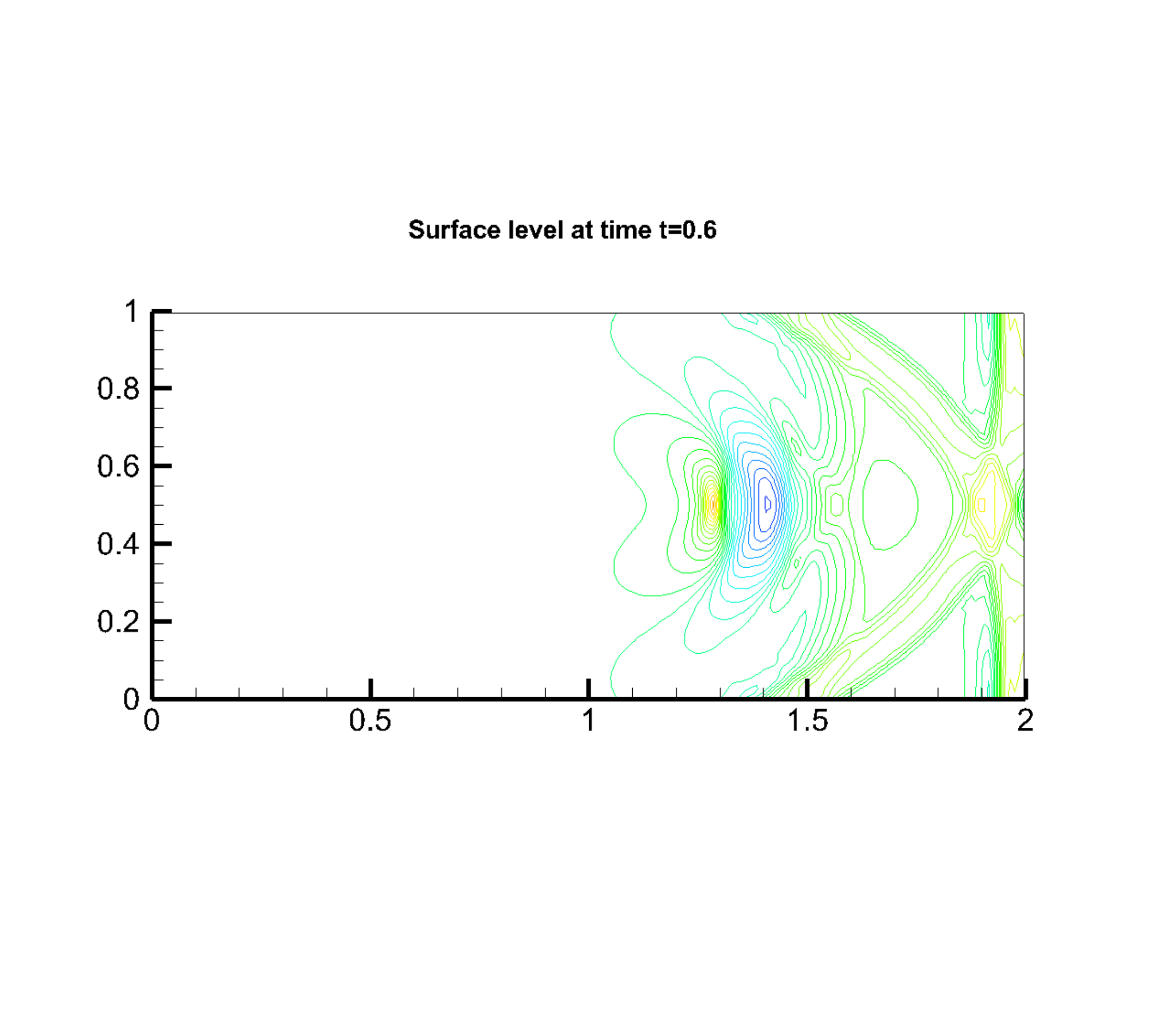} }
\subfigure{
\includegraphics[width=0.45\textwidth]{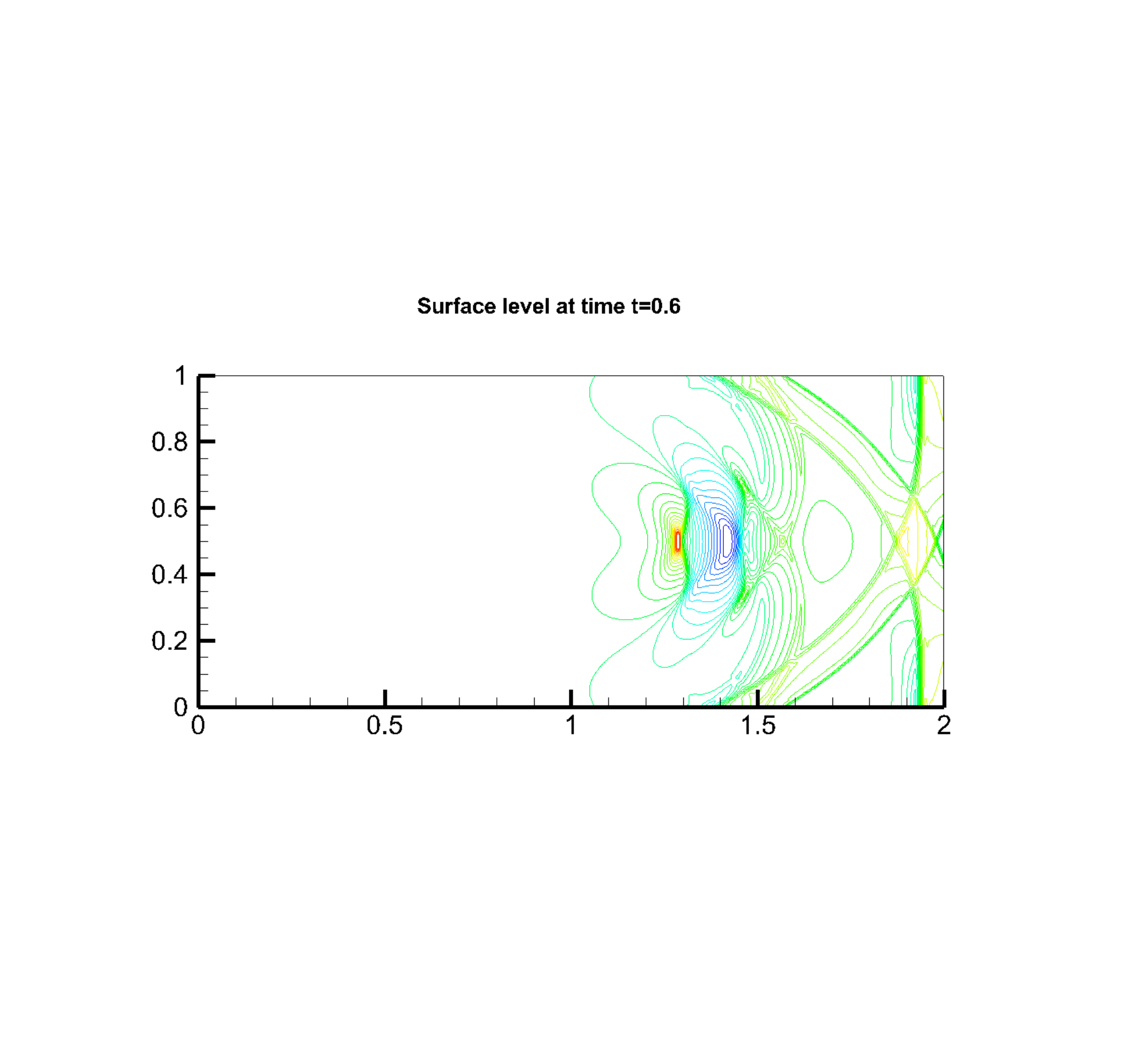} }
\end{figure}

\section{Concluding remarks}
\label{sec_sum}

In this paper, we developed a well-balanced oscillation-free discontinuous Galerkin (OFDG)
method for solving the shallow water equations.
Following the idea of the OFDG method in \cite{LLSSISC2021, LLSSINUM2021}, we added the
suitable extra damping terms to the existing well-balanced DG schemes proposed in
\cite{XSJCP2006, XSCICP2006}. The extra damping terms are carefully designed so as to
achieve the well-balanced property. It indicates the damping terms in the OFDG method
is very flexible and they can be consistent with other good properties with some
suitable modifications. The numerical experiments
validated the proposed method had well performances for several benchmark problems.
In our future plan, we will extend the current algorithm to the moving water steady state problems
and other well-balanced dynamics such as the hyperbolic model of chemosensitive
movement and the compressible Euler equations with gravitation, etc.

\end{document}